\documentclass[a4paper,11pt]{amsart}

\usepackage[utf8]{inputenc}
\usepackage[T1]{fontenc}
\usepackage{lmodern}
\usepackage[english]{babel}
\usepackage{amsmath}
\usepackage{amssymb}
\usepackage{mathrsfs}
\usepackage{amsthm}
\usepackage{esint}
\usepackage[abbrev]{amsrefs}
\usepackage{cleveref}
\usepackage{tikz} 
\usepackage{pgfplots}

\makeatletter
\def\namedlabel#1#2{\begingroup
   \def\@currentlabel{#2}%
   \label{#1}\endgroup
}
\makeatother

\usepackage{pgf,tikz}
\usetikzlibrary{arrows}
\usetikzlibrary[patterns]

\usepackage[left=2.5cm,right=2.5cm,top=4.5cm,bottom=3.5cm]{geometry}
\theoremstyle{plain}
\newtheorem{proposition}{Proposition}[section]

	\crefname{IntroductionClaim}{proposition}{propositions}
\newtheorem{lemma}{Lemma}[section]

\newtheorem{theorem}{Theorem}[section]
\newtheorem*{theorem*}{Theorem}

\theoremstyle{definition}

\newtheorem{definition}{Definition}[section]
\theoremstyle{remark}
\newtheorem{remark}{Remark}[section]

\numberwithin{equation}{section}

\newcounter{tmp}
\newcounter{constantCounter}
\newcommand{\cste}{{C_{\arabic{constantCounter}}}}
\newcommand{\ncste}{\stepcounter{constantCounter}\cste}
\expandafter\let\expandafter\oldproof\csname\string\proof\endcsname
\let\oldendproof\endproof
\renewenvironment{proof}[1][\proofname]{%
  \oldproof[\setcounter{constantCounter}{0}{#1}.]%
}{\setcounter{constantCounter}{0}\oldendproof}
	
\title[Asymptotic of steady vortex pairs in the lake equation]{Asymptotic of steady vortex pair\\in the lake equation}
\date{February 9, 2018}
\subjclass[2010]%
{76C05 (35Q35,35C15,58E30)}

\author{Justin Dekeyser}
\address[Justin Dekeyser]{Université~catholique~de~Louvain,
Département~de~Mathématique,\newline
2~Chemin~du~Cyclotron, 1348~Louvain-La-Neuve, \textsc{Belgium}}
\email[Justin Dekeyser]{Justin.Dekeyser@uclouvain.be}

\newcommand{\reals}{\mathbb{R}}
\newcommand{\plane}{\reals^2}
\newcommand{\integers}{\mathbb{N}}
\DeclareMathOperator{\du}{d}
\newcommand{\dif}{{\,\du}}
\newcommand{\scalarproduct}[3]{({#1}|{#2})_{#3}}
\newcommand{\norm}[2]{\Vert{#1}\Vert_{#2}}

\newcommand{\domain}{\Omega}
\newcommand{\depth}{\textrm{b}}
\newcommand{\closure}[1]{\overline{#1}}
\newcommand{\weakclosure}[1]{\closure{#1}^{\textrm{w}}}
\newcommand{\curve}{\mathcal{C}}
\newcommand{\nbislands}{m}

\newcommand{\lebesguemeasure}{\textrm{m}}
\newcommand{\mumeasure}{\dif\mu}
\newcommand{\gradient}{\nabla}
\newcommand{\flipgradient}{\gradient^\perp}
\newcommand{\largefunctionspace}{\mathcal{H}}
\newcommand{\functionspace}{\largefunctionspace_0}

\newcommand{\divergence}{\textrm{div}}
\newcommand{\curl}{\textrm{curl}}
\newcommand{\circulation}{\textrm{c}}
\newcommand{\staticflow}{\psi}
\newcommand{\positivepart}[1]{\left(#1\right)_+}
\newcommand{\negativepart}[1]{\left(#1\right)_-}

\newcommand{\vortex}{\zeta}
\newcommand{\vortexoperator}{\mathfrak{K}}
\newcommand{\rectifycirculation}{\mathfrak{H}}
\newcommand{\boundary}{\partial}

\newcommand{\laplacian}{-\Delta}
\newcommand{\greenlaplace}{\textrm{g}}
\newcommand{\rectifykernel}{\textrm{R}}
\newcommand{\regulargreenlaplace}{\textrm{H}}
\newcommand{\distance}{\textrm{d}}
\newcommand{\diameter}{\textrm{diam}}
\newcommand{\energy}{\textrm{E}}

\newcommand{\rearrangement}{\textrm{Rearg}}

\newcommand{\symmetrizearoundpoint}[2]{{#1}\sharp{#2}}

\newcommand{\vortexstrength}{\mathcal{S}}
\newcommand{\threshold}{\varsigma}
\newcommand{\positivefirstorderflow}[1]{\mathcal{T}^+_{#1}}
\newcommand{\negativefirstorderflow}[1]{\mathcal{T}^-_{#1}}
\newcommand{\error}{\Theta}

\newcommand{\distribution}[1]{\textrm{D}\left[{#1}\right]}
\newcommand{\scale}[1]{\mathfrak{s}{#1}}
\newcommand{\scalemumeasure}{\dif\mathfrak{s}\mu}

\begin{document}

\begin{abstract}
We bring new results in the study the asymptotic behavior of shrinking vortex pairs
obtained by maximization of the kinetic energy in a 2-dimensional lake over a class of rearrangements.
After improving recent results obtained for the first order asymptotic behavior of such pairs,
we focus on second order asymptotic properties. We show that among all points of maximal depth, the vortex locates according to
an adaptation of the Kirchoff-Routh function, and we study the asymptotic shape of optimal vortices. We also explore a relaxed maximization
problem with uniform constraints, for which we prove that the distribution consists of two vortex patches.
\end{abstract}
\maketitle

%\tableofcontents 

\section{Introduction}

\begingroup
\setcounter{tmp}{\value{theorem}}
\renewcommand\thetheorem{\Alph{theorem}}

The aim of this paper is to bring new results in the study of the asymptotic behavior of steady vortex pairs
in the lake equations. The \emph{steady} lake equations in their velocity-pressure formulation read as~\cite{CamassaHolmLevermore}
	\[ \begin{cases}
		\divergence\big( \depth u \big) = 0 ,&\text{on}\ \domain\subseteq\plane, \\
		(u\cdot\gradient)u=\gradient p ,&\text{on}\ \domain,\\
		\displaystyle \oint_{\curve_i}\scalarproduct{u}{\tau_i}{\plane} = \circulation_i,&\forall i\in\{0,\dots,\nbislands\} ,
	\end{cases} \]
where $\curve_0,\dots,\curve_\nbislands$ are the connected components of $\boundary\domain$ and $\circulation_i$ is the circulation along $\curve_i$.
The function $\depth:\domain\to\reals^+$ is referred as the \emph{depth function}.
Let
	\[ \energy(\vortex) = \frac{1}{2}\int_\domain|u|^2\mumeasure = \frac{1}{2}\int_\domain\vortex\vortexoperator\vortex\mumeasure ,\qquad \mumeasure(x)=\depth\dif x ,\]
be the kinetic energy of the system, where $\vortex=\depth^{-1}\curl(u)$ is the \emph{potential vortex field} and
$\vortexoperator\vortex$ is the \emph{stream function} defined as the solution of the elliptic problem:
	\[ \begin{cases}
		-\divergence\Big( \depth^{-1}\gradient\vortexoperator\vortex\Big) = \depth\vortex,\\
		\displaystyle \oint_{\curve_i}\scalarproduct{\depth^{-1}\flipgradient\vortexoperator\vortex}{\tau_i}{\plane} = \circulation_i, \quad\forall i\in\{0,\dots,\nbislands\} .
	\end{cases} \]
The velocity field $u$ is linked to the stream function by mean of the identity $u:=\depth^{-1}\flipgradient\vortexoperator\vortex$,
while the potential vortex satisfies the transport equation
	\[ \partial_t\vortex + \scalarproduct{u}{\gradient\vortex}{\plane} = 0 .\]
This suggests that the measure $\mu$ is the relevant invariant measure for the problem.
In fact, it is known that \emph{steady} weak solutions of the lake problem may be obtained by energy maximization over the set of all $\mu$-rearrangements
of some given potential vortex
(see~\cite{myArticle}, and \citelist{\cite{BurtonRearrangementOfFunctions}\cite{BurtonSteadyConfiguration}\cite{BurtonVariationalProblems}}
for similar questions for the 3D Euler equations with axis-symmetry). To state it in a more concrete fashion, let us consider
a given distribution function
	\[ D : \reals^+\to [0,\mu(\domain)] \]
such that
	\[ \int_{\reals^+}D(\lambda)\dif\lambda = 1 \]
and, for some $p>1$,
	\[ \int_{\reals^+}\lambda^p\ D(\lambda)\dif\lambda < +\infty .\]
We are interested in potential vortex fields $\vortex_\epsilon$ that satisfy the distributional conditions
	\begin{equation}\label{distributionConstraint}
		\mu\big(\{\positivepart{\vortex_\epsilon}\geq\lambda\} \big) = 
		\frac{\epsilon^2}{\delta}\,D\bigg( \frac{\epsilon^2\lambda}{\delta\tau\vortexstrength_\epsilon} \bigg),
		\qquad
		\mu\big(\{\negativepart{\vortex_\epsilon}\geq\lambda\} \big) = 
		\frac{\epsilon^2}{\delta}\,D\bigg( \frac{\epsilon^2\lambda}{\delta(1-\tau)\vortexstrength_\epsilon} \bigg) ,
	\end{equation}
where $\delta=\sup\limits_{\lambda>0}D(\lambda)$, $\tau\in[0,1]$ and $\vortexstrength_\epsilon>0$.
In these settings, the positive part $\positivepart{\vortex_\epsilon}$ and the negative part $\negativepart{\vortex_\epsilon}$
of the vortex would satisfy
	\[ \mu\big( \{\positivepart{\vortex_\epsilon}>0\} \big) \leq \epsilon^2,
		\qquad \mu\big(\{\negativepart{\vortex_\epsilon}>0\}\big) \leq \epsilon^2, \]
and we would also have the vortex-strength prescriptions
	\[ \int_\domain\positivepart{\vortex_\epsilon}\mumeasure = \tau\vortexstrength_\epsilon ,
		\qquad \int_\domain\negativepart{\vortex_\epsilon}\mumeasure = (1-\tau)\vortexstrength_\epsilon .\]
The following result is the starting point of our analysis:
\begingroup
\setcounter{theorem}{1} %assign desired value to theorem counter
\renewcommand\thetheorem{(\fnsymbol{theorem})}
\begin{theorem}[\cite{myArticle}*{Theorem~3.1}]\label{strongConcentrationReference}
Let $(\domain,\depth)$ be a lake with $\domain\subseteq\plane$ bounded and of class $C^1$, and $\depth\in C^1(\closure{\domain})$
with $\inf_\domain\depth>0$, or $\depth=\phi^\alpha$ with $\alpha>0$ and $\phi$ a regularization of the distance at the boundary.
Let $\big\{\vortex_\epsilon\in L^p(\domain,\mu):\epsilon>0\big\}$ be a family of solutions of the steady lake equations
obtained by energy maximization over constraint~\eqref{distributionConstraint}.

The accumulation points of $\big\{\positivepart{\vortex_\epsilon}:\epsilon>0\big\}$
and $\big\{\negativepart{\vortex_\epsilon}:\epsilon>0\big\}$ are Dirac masses
both centered around a point of maximal depth.
\end{theorem}
\endgroup
\setcounter{theorem}{0} %assign desired value to theorem counter
In order to prove \cref{strongConcentrationReference}, the author used the following integral kernel representation
of the stream function:
	\[ \vortexoperator\vortex(x)
		= \frac{\depth(x)}{2\pi}\int_\domain\frac{\diameter(\domain)}{|x-y|}\vortex(y)\mumeasure(y)
			+ \int_\domain F(x,y)\vortex(y)\mumeasure(y) ,\]
where the function $F$ is a correction function
depending on $\depth$ and the circulation conditions.
Relying on the integral kernel expansion of the stream function,
the vortex pair was proved to be of the form $\vortex_\epsilon=\vortex_\epsilon\chi_{D_\epsilon} + \vortex_\epsilon\chi_{\domain\setminus D_\epsilon}$,
where $\vortex_\epsilon\chi_{D_\epsilon}$ tends to a singular vortex pair in the limit $\epsilon\to 0$, while
	\[ \int_{\domain\setminus D_\epsilon}\vortex_\epsilon\mumeasure \sim \vortexstrength_\epsilon\,\frac{1}{\log\frac{1}{\epsilon}} .\]
We recall that the leading term in the energy $\energy_\epsilon(\vortex_\epsilon)$ grows like
	\[ \norm{\vortexoperator\vortex_\epsilon}{\infty}\int_\domain\vortex_\epsilon\mumeasure \sim \vortexstrength_\epsilon^2\log\frac{1}{\epsilon} ,\]
while the relevant second order terms are of order $\vortexstrength_\epsilon^2$.
Thus, if one wants to obtained a more accurate picture of the asymptotic behavior of the pair by
going beyond the leading term, one should try to obtain estimates for the $\vortexstrength_\epsilon^2$-order terms.

Unfortunately, such estimates were hard to derive with the purely integral-comparison techniques previously used.
Indeed, we could naively try to compare the optimal energy $\energy_\epsilon(\vortex_\epsilon)$
with the energy of some nearly-spherical competitor $(\vortex_\epsilon\chi_{D_\epsilon})^\star+\vortex_\epsilon\chi_{\domain\setminus D_\epsilon}$
obtained by some symmetrization technique. Although such competitor would be suitable for estimations,
the error we commit is estimated as
	\[ \energy_\epsilon(\vortex_\epsilon) - \energy_\epsilon(\vortex_\epsilon^\star)
		\lesssim \norm{\vortexoperator\vortex}{\infty}\ \int_{\domain\setminus D}\vortex\mumeasure
		\sim \Bigg( \int_\domain\vortex\mumeasure\Bigg)^2 \sim \vortexstrength_\epsilon^2 . \]
Since the error of such a comparison process could not be estimated of lower order than $\vortexstrength_\epsilon^2$, and because
the second order relevant terms in the energy expansion $\energy_\epsilon(\vortex_\epsilon)$ are of order $\vortexstrength_\epsilon^2$,
the analysis could not be performed further.

\subsection*{Statement of the results}
In view of the above discussion, a natural strategy is to improve the concentration result obtained in~\cite{myArticle}.
This is the first step of the paper:
\begin{theorem}\label{introductionEssentialConcentration}
In the settings of \cref{strongConcentrationReference}, there exists $\varsigma\in(0,1]$
such that
	\[ \limsup_{\epsilon\to0}\frac{\diameter\big(\{\positivepart{\vortex_\epsilon}>0\}\big)}{\epsilon^\varsigma} < +\infty ,\]
and
	\[ \limsup_{\epsilon\to0}\frac{\diameter\big(\{\negativepart{\vortex_\epsilon}>0\}\big)}{\epsilon^\varsigma} < +\infty .\]
\end{theorem}
In comparison with \cref{strongConcentrationReference}, \cref{introductionEssentialConcentration} claims that concentration occurs essentially is two large balls
of decreasing radius. With this result, we are able to study the second order behavior of the vortex pair:
\begin{theorem}\label{introductionSecondOrder}
In the settings of \cref{strongConcentrationReference}, let $F:\domain\times\domain\to\reals$ be the function defined in the Green's function expansion,
and let $G:\domain\times\domain\to\reals$ be the function defined by
	\[ G(x,y) = \frac{\depth(x)+\depth(y)}{2\pi}\log\frac{\diameter(\domain)}{|x-y|} .\]
Assume that $\depth$ admits at least two maximizers $X,Y$ in $\domain$ such that
	\[ \liminf_{\epsilon\to 0}\distance\big( \{\positivepart{\vortex_\epsilon}>0 \}, X \big) = 0 ,
		\qquad \liminf_{\epsilon\to 0}\distance\big( \{\negativepart{\vortex_\epsilon}>0 \}, Y \big) = 0 . \]
Then $(X,Y)$ minimizes the function
	\begin{align*}
		(x,y)&\in\domain\times\domain
		\\&\mapsto \bigg[\tau(1-\tau)G(x,y)
		-\tau^2F(x,x) - (1-\tau)^2F(y,y) +\tau(1-\tau)\big(F(x,y)+F(y,x)\big)\bigg]
	\end{align*}
over the set $\big(\domain\cap\{\depth=\sup_\domain\depth\}\big)\times\big(\domain\cap\{\depth=\sup_\domain\depth\}\big)$.
\end{theorem}
Note that when $\depth\equiv 1$, every point in $\domain$ is a point of maximal depth.
Moreover, when $\depth\equiv 1$, we have $\gradient\depth=0$
and the correction function $F$ in our Green's function expansion turns out to be the Kirchoff-Routh function,
which rules the motion of singular vortex pairs in the 2D Euler equations~\citelist{\cite{Lin1}\cite{Lin2}}.
Although the function $\rectifykernel$ may be hard to compute in the general setting of a lake $(\domain,\depth)$,
\cref{introductionSecondOrder} fills a conceptual gap between the lake model and the 2D Euler equations.

We are also in position to investigate the asymptotic shape of the optimal vortices.
Relying on tools from standard potential theory and on an \emph{asymptotic version of the Riesz-Sobolev rearrangement inequality},
we prove the following reuslt:
\begin{theorem}\label{introductionShape}
In the setting of \cref{strongConcentrationReference}, and if $\domain$ satisfies an
interior cone condition, then there exists translations $\big\{T_\epsilon:\plane\to\plane\big\}$
such that every accumulation point as $\epsilon\to 0$, in the sense of convergence in measure, of the
rescaled positive parts
	\[ \Bigg\{
	\frac{\epsilon}{\tau\vortexstrength_\epsilon}\ \positivepart{\vortex_\epsilon}\big(\epsilon\ T_\epsilon\big)
	: \epsilon > 0
	\Bigg\} \]
are radially symmetric functions. Similarly, there exists translations $\big\{T'_\epsilon:\plane\to\plane\big\}$
such that every accumulation points as $\epsilon\to 0$, in the sense of convergence in measure,
of the rescaled negative parts
	\[ \Bigg\{
	\frac{\epsilon}{(1-\tau)\vortexstrength_\epsilon}\ \negativepart{\vortex_\epsilon}\big(\epsilon\ T'_\epsilon\big)
	: \epsilon > 0
	\Bigg\} \]
are radially symmetric functions.
\end{theorem}
A final topic we would like to focus on a relaxed maximization problem,
where the distribution function $D$ is not known a priori but $L^\infty$-constraints are provided.
As a consequence of the previous results, we will be able to prove the following:
\begin{theorem}\label{introductionBestDistribution}
Let $\Gamma_\epsilon$ be the set of functions defined by
	\begin{multline*}
		\Gamma_\epsilon = \Big\{ \vortex: 0\leq \positivepart{\vortex}\leq \epsilon^{-2}\tau\vortexstrength_\epsilon,
			0\leq \negativepart{\vortex}\leq \epsilon^{-2}(1-\tau)\vortexstrength_\epsilon;
			\\ \int_\domain\positivepart{\vortex}\mumeasure=\tau\vortexstrength_\epsilon,
				\int_\domain\negativepart{\vortex}\mumeasure=(1-\tau)\vortexstrength_\epsilon
	 \Big\} . \end{multline*}
Every energy $\energy_\epsilon$ maximizer $\vortex_\epsilon$ over $\Gamma_\epsilon$ is necessarily of the form
	\[ \vortex_\epsilon = \epsilon^{-2}\tau\vortexstrength_\epsilon\chi_{\{\vortex_\epsilon>0\}}
		- \epsilon^{-2}(1-\tau)\vortexstrength_\epsilon\chi_{\{\vortex_\epsilon<0\}} \]
and
	\[ \mu\big( \{\vortex_\epsilon\neq 0\}\big) \leq \epsilon^2 .\]
Furthermore, the sets $\{\vortex_\epsilon>0\}$ and $\{\vortex_\epsilon<0\}$ are asymptotically close to balls.
\end{theorem}
A similar problem was studied by Turkington~\cite{TurkingtonSteady1} and Turkington~\&~Friedmann,
with the maximization class
	\begin{equation*}
		\Gamma_\epsilon = \Big\{ \vortex: 0\leq \vortex\leq \epsilon^{-2}\vortexstrength_\epsilon;
			\int_\domain\vortex\mumeasure=\tau\vortexstrength_\epsilon \Big\} .
	\end{equation*}
Relying on potential theory techniques and symmetrization arguments, the authors managed to prove that maximizers
should be a vortex patch and, actually, a ball. Although these techniques are available only in a limit regime,
we prove a similar optimal-distribution result, relying on the convex structures of rearrangements.

\endgroup
\setcounter{theorem}{\thetmp}

\subsection*{Organization of the paper}
We begin by recalling the framework we are going to work in. Results are mainly cited from~\cite{myArticle}.

In a second section, we prove \cref{introductionEssentialConcentration} relying on the differential structure of the problem,
following ideas of Turkington~\cite{TurkingtonSteady1} and Elcrat~\&~Miller~\cite{ElcratMiller}.

The third section is devoted to the study of repulsion effects acting on the vortex pairs.
We prove that the vortex pair cannot come close to each others too fast, and their distance to the boundary $\boundary\domain$
remains small in comparison to relevant asymptotic orders. We also prove \cref{introductionSecondOrder}.
The techniques of proof are purely integral comparison arguments, based on a variant of the Sobolev-Riesz rearrangement inequality
and our integral kernel expansion of the stream function.

In the fourth section, we study the family of rescaled vortices and we prove that every accumulation point of the rescaled versions,
in the sense of convergence in measure, are symmetric functions. This section is based on a previous work of Burchard~\&~Guo~\cite{BurchardGuo}.

The last section is devoted to the relaxed maximization problem. We rely on the convex structure of sets of rearrangements,
already studied by Ryff~\citelist{\cite{Ryff1}\cite{Ryff2}} and Burton~\citelist{\cite{BurtonRearrangementOfFunctions}\cite{BurtonSteadyConfiguration}\cite{BurtonVariationalProblems}}.

\subsection*{Acknowledgment}
The author thanks Jean~Van~Schaftingen for their many discussions during the elaboration of this manuscript.

\section{Framework}

In this section we recall the framework used in~\cite{myArticle} for the energy maximization problem,
and we recall properties of maximizing vortex pairs.
We say that two measurable functions $f,g:\domain\to\reals$ are $\mu$-rearrangements of each others if, for all $\lambda\in\reals$, we have
	\[ \mu\big( \{ f\geq \lambda \} \big) = \mu\big( \{g\geq \lambda\} \big) .\]
Here, the measure $\mumeasure(x)=\depth(x)\dif x$ is the Lebesgue measure weighted by the depth function $\depth$
of the lake. We recall that the depth function $\depth\in \bigcup_{\alpha>0}C^{0,\alpha}(\closure{\domain})$ is assumed to be Hölder continuous 
up to the boundary, and positive on compact subsets of $\domain$.
On $\domain$ we assume that there exists disjoint connected compact sets $\curve_1,\dots,\curve_\nbislands$ such that
	\[ \domain = \closure{\domain_0}\setminus\bigcup^\nbislands_{i=0}\curve_i ,\]
with $\domain_0$ a bounded open connected satisfying $\domain_0=\text{int}\big(\closure{\domain_0}\big)$,
$\curve_0=\boundary{\domain_0}$.

We are going to consider (measurable) potential vortices $\vortex_\epsilon$ parametrized by $\epsilon>0$ through
	\[ \mu\big( \{\vortex_\epsilon\neq 0\} \big)  = \epsilon^2 ,\]
and satisfying the integral identities
	\[ \int_\domain\positivepart{\vortex_\epsilon}\mumeasure = \tau\vortexstrength_\epsilon ,
		\quad \int_\domain\negativepart{\vortex_\epsilon}\mumeasure = (1-\tau)\vortexstrength_\epsilon ,\]
where $\vortexstrength_\epsilon>0$ and $\tau\in(0,1)$. As in~\cite{myArticle}, we also make the assumption that there exists $p>1$ such that
	\begin{equation}\label{distribConstraint}
	\sup_{\epsilon>0}\Bigg\{
		\frac{\norm{\positivepart{\vortex_\epsilon}}{L^p_{\mu}}\epsilon^{2(1-\frac{1}{p})}}{\tau\vortexstrength_\epsilon}
		+ \frac{\norm{\negativepart{\vortex_\epsilon}}{L^p_{\mu}}\epsilon^{2(1-\frac{1}{p})}}{(1-\tau)\vortexstrength_\epsilon}
	\Bigg\} < +\infty .
	\end{equation}
The limit cases $\tau=0$ and $\tau=1$ are also possible, although they demand more writing cautions.
These limit cases represent the cases of a single non sign changing vortex. Observe that by construction, vortices
constructed according to~\eqref{distributionConstraint} also satisfy $L^p(\domain,\mu)$ constraint~\eqref{distribConstraint}.

The \emph{stream function} associated with some potential vortex $\vortex_\epsilon$ is the solution of the elliptic equation
	\begin{equation}\label{ellipticproblem}
		\begin{cases}
			-\divergence\big( \depth^{-1}\gradient \psi \big) = \depth\vortex_\epsilon ,\\
			\displaystyle \oint_{\boundary\curve_i}\scalarproduct{\depth^{-1}\flipgradient\phi}{\tau_i}{\plane} = \circulation_i(2\tau-1)\vortexstrength_\epsilon ,
		\end{cases}
	\end{equation}
where $\circulation_0,\dots,\circulation_\nbislands$ are real numbers such that, in view of Kelvin's theorem,
	\[ \sum_{i=0}^\nbislands\circulation_i = 1 .\]
As it was proved in~\cite{myArticle}, this elliptic problem~\eqref{ellipticproblem} has a unique solution in some Sobolev space.
Indeed, let us define the vector space
	\[ \largefunctionspace = \Big\{ f\in W^{1,2}(\domain) : \depth^{-1}|\gradient f|^2\in L^1(\domain,\lebesguemeasure) \Big\} ,\]
which we endow with the scalar product
	\[ \scalarproduct{f}{g}{\largefunctionspace} = \int_\domain\frac{\scalarproduct{\gradient f}{\gradient g}{\plane}}{\depth}\ \lebesguemeasure + \scalarproduct{f}{g}{W^{1,2}} .\]
The pair $\big(\largefunctionspace,\scalarproduct{\cdot}{\cdot}{\largefunctionspace}\big)$ is a Hilbert space,
whose elements induce a finite ``lake-energy''
	\[ \int_\domain\big(\depth^{-1}\norm{\flipgradient\phi}{\plane}\big)^2\ \mumeasure < +\infty .\]
Since $\depth$ is positive on compact sets, the collection $\mathscr{C}_c$ of those functions in $C^1(\closure{\domain})$
that are constant on a neighborhood of $\boundary\domain$, belongs to $\largefunctionspace$.
We write $\functionspace$ the closure of $C^1_c(\domain)$ in $\largefunctionspace$, and we denote by $\largefunctionspace\times\largefunctionspace$ the bilinear rule
	\[ (f,g)\in\largefunctionspace\mapsto
		\scalarproduct{f}{g}{\functionspace} = \int_\domain\frac{\scalarproduct{\gradient f}{\gradient g}{\plane}}{\depth}\ \lebesguemeasure .\]
Since $\domain$ is bounded,
one may rely on standard Poincaré's inequality to see that the latter defines an equivalent scalar product on $\functionspace$.
Whenever the lake $(\domain,\depth)$ enjoys sufficiently regularity properties,
the elliptic problem~\eqref{ellipticproblem} admits a weak solution of the form~\cite{myArticle}
	\begin{equation}\label{integralRepresentation}
		\psi(x)
			= \frac{\depth(x)}{2\pi}\int_\domain\log\frac{\diameter(\domain)}{|x-y|}\vortex(y)\mumeasure(y)
			+ \int_\domain F(x,y)\vortex(y)\mumeasure(y) ,
	\end{equation}
where the function
$F:\domain\times\domain\to\reals$ is defined for all $x,y\in\domain$ by
	\[ F(x,y) = \rectifykernel(x,y) - \depth(x)\regulargreenlaplace(x,y)
		+ \sum^{\nbislands}_{i=0}(\staticflow_i(x)-\circulation_i)\Big[
		\mathcal{A}^{-1} \big[\staticflow_j(y)-\circulation_j\big]_{0\leq j\leq \nbislands}
		\Big]_i ,\]
and the functions $\rectifykernel,\regulargreenlaplace,\mathcal{A}$ and $\staticflow_0,\dots,\staticflow_{\nbislands}$
are defined as follows:
\begin{itemize}
	\item the function $\regulargreenlaplace$ is defined for all $x,y\in\domain$ by
	\[ \regulargreenlaplace(x,y)
		= \frac{1}{2\pi}\log\frac{\diameter(\domain)}{|x-y|}
		- \greenlaplace(x,y) ,\]
	where $\greenlaplace$ is the Green's function associated to the Laplace's
	operator $\laplacian$ in $\domain$ with Dirichlet boundary conditions;
	\item for all $y\in\domain$, the function $\rectifykernel(\cdot,y)$ belongs to $\functionspace$ and for all $\varphi\in\functionspace$, we have
		\[ \int_\domain\scalarproduct{\gradient\rectifykernel(\cdot,y)}{\gradient\varphi}{\plane}\ \frac{\dif x}{\depth}
			= \int_\domain\scalarproduct{\greenlaplace(\cdot,y)\gradient \depth}{\gradient\varphi}{\plane}\ \frac{\dif x}{\depth} ,\]
	the function
		\[ y\in\domain\mapsto \norm{\rectifykernel(\cdot,y)}{L^\infty} \]
	admits an uniformly continuous extension to $\closure{\domain}$, and the function
		\[ (x,y)\in\domain\times\domain \mapsto \rectifykernel(x,y) \]
	is measurable.
	\item for all $i\in\{0,\dots,\nbislands\}$, the function
	$\staticflow_i\in L^\infty(\domain)\cap C(\domain)\cap\closure{\mathscr{C}_c}$ is the limit in $\largefunctionspace$
	of functions in $\mathscr{C}_c$ that equal $\delta_{ij}$ on a neighborhood of $\curve_j$, with $\delta_{ij}$
	the Kronecher symbol; and $\staticflow_i$ is bounded by $1$ and satisfies
		\[ \scalarproduct{\staticflow_j}{\varphi}{\functionspace} = 0,\quad\text{\normalfont for all}\ \varphi\in\functionspace .\]
	\item the functions $\staticflow_0,\dots,\staticflow_\nbislands$ are linearly independent, we have
		\[ \closure{\mathscr{C}_c} = \closure{\functionspace \oplus \text{\normalfont Vec}\langle\staticflow_0,\dots,\staticflow_\nbislands\rangle} ,\]
	and
		\[ \sum^\nbislands_{i=0}\staticflow_i = 1 \]
	on $\domain$. Moreover, the operator
		\[ \mathcal{A} : \Big\{ (v_0,\dots,v_{\nbislands})\in\reals^{\nbislands+1}:
			\sum^\nbislands_{i=0}v_i = 0 \Big\}
		\longrightarrow \Big\{ (\alpha_0,\dots,\alpha_{\nbislands}) : \alpha_0 = 0 \Big\} ,\]
	is a linear isomorphism.
\end{itemize}
The existence of such integral representation~\eqref{integralRepresentation}
for the stream function $\psi$ is proved in the appendix of~\cite{myArticle}, provided the lake $(\domain,\depth)$
is regular enough.
\begin{definition}
A lake $(\domain,\depth)$ is said to be \emph{continuous} if the operator $\vortexoperator$
admits the integral representation~\eqref{integralRepresentation}, with $F$ as above.
\end{definition}
The main examples for a lake $(\domain,\depth)$ to be continuous in the above sense are twofold.
First, one could assume $b\in W^{1,\infty}(\domain)\cap \bigcup_{\alpha>0}C^{0,\alpha}(\closure{\domain})$
with the additional condition that $\inf_\domain\depth>0$. A second example would be the situation where
there exists a regularization of the distance at the boundary $\phi\in C^1$ and some $\alpha>0$
such that $\depth=\phi^\alpha$. Mixed conditions are also possible,
and one could reduce the regularity on $\domain$ provided the distance at the boundary is
replaced by an appropriate quantity. For more details, we refer to~\cite{myArticle}.

The energy $\energy_\epsilon$ we are interested in takes the equivalent form
	\begin{multline*}
		\energy_\epsilon(\vortex_\epsilon) = \iint_{\domain\times\domain}\depth(x)\log\frac{\diameter(\domain)}{|x-y|}\ \vortex_\epsilon(x)\vortex_\epsilon(y)
		\mumeasure\mumeasure(x,y)
			\\ + \iint_{\domain\times\domain}F(x,y)\ \vortex_\epsilon(x)\vortex_\epsilon(y)\ \mumeasure\mumeasure(x,y) .
	\end{multline*}

\section{Improvement of leading order asymptotic}

We introduce two important quantities, referred as the \emph{first order flows}.
The first order flows associated with a vortex $\vortex\in L^p(\domain,\mu)$ are defined as
	\[ \positivefirstorderflow{\vortex}(x) = \frac{\depth(x)}{4\pi}\int_\domain\log\frac{\diameter(\domain)}{|x-y|}\positivepart{\vortex}(y)\mumeasure(y) ,\]
and
	\[ \negativefirstorderflow{\vortex}(x) = \frac{\depth(x)}{4\pi}\int_\domain\log\frac{\diameter(\domain)}{|x-y|}\negativepart{\vortex}(y)\mumeasure(y) .\]
For $\vortex\in L^p(\domain,\mu)$, the positive first order flow
$\positivefirstorderflow{\vortex}$ induced by $\vortex$ satisfies
	\[ \positivefirstorderflow{\vortex}(x)
		\leq \frac{\sup_\domain\depth}{2\pi}\log\frac{\diameter(\domain)}{\epsilon}\int_\domain\positivepart{\vortex}\mumeasure
		+ C\norm{\vortex}{L^p(\mu)}\epsilon^{2\big(1-\frac{1}{p}\big)} .\]
Indeed, we have since $\positivepart{\vortex}$ is positive
	\begin{align*}
		\int_\domain\log\frac{\epsilon}{|x-y|}\positivepart{\vortex}(y)\mumeasure(y)
		&\leq \int_\domain\positivepart{\log\frac{\epsilon}{|x-y|}}\positivepart{\vortex}(y)\mumeasure(y)
		\\&\leq \norm{\positivepart{\vortex}}{L^p_\mu}\Bigg(
			\int_\domain\positivepart{\log\frac{\epsilon}{|x-y|}}^{\frac{p}{p-1}}\mumeasure(y)
		\Bigg)^{1-\frac{1}{p}} ,
	\end{align*}
and a change of variable now yields to
	\[ \int_\domain\log\frac{\epsilon}{|x-y|}\positivepart{\vortex}(y)\mumeasure(y)
		\leq \ncste\norm{\positivepart{\vortex}}{L^p_\mu}\Bigg(\epsilon^2
			\int_{B(0,1)}\positivepart{\log\frac{1}{|y|}}^{\frac{p}{p-1}}\dif y
		\Bigg)^{1-\frac{1}{p}} .\]
Combining the above a priori estimate with the control condition~\eqref{distribConstraint},
we conclude the following upper estimates for the first order flows:
\begin{lemma}\label{upperBoundFirstOrder}
There exists a constant $C>0$ such that
for all $\epsilon>0$ and for all $\vortex\in\rearrangement(\vortex_\epsilon)$,
we have for all $x\in\domain$:
	\[ \positivefirstorderflow{\vortex}(x)
		\leq \frac{\tau\sup_\domain\depth}{2\pi}\vortexstrength_\epsilon
			\log\frac{1}{\epsilon} + C\vortexstrength_\epsilon ,\]
and
	\[ \negativefirstorderflow{\vortex}(x)
		\leq \frac{(1-\tau)\sup_\domain\depth}{2\pi}\vortexstrength_\epsilon
			\log\frac{1}{\epsilon} + C\vortexstrength_\epsilon .\]
\end{lemma}
Next we have the following concentration theorem:
\begingroup
\setcounter{theorem}{1} %assign desired value to theorem counter
\renewcommand\thetheorem{(\fnsymbol{theorem})}
\begin{theorem}[\cite{myArticle}*{Corollary~3.1, Theorem~3.1, Theorem~3.2}]\label{weakConcentration}
Let $(\domain,\depth)$ be a regular lake.
Let $\big\{\vortex_\epsilon\in L^p(\domain,\mu):\epsilon>0\big\}$ be a family of solutions of the steady lake equations
obtained by energy maximization over their set of $\mu$-rearrangements,
with constrained~\eqref{distribConstraint}.

There exists $\kappa>0$, $\threshold\in(0,1]$ and a family $\big\{\error_\epsilon:\epsilon>0\big\}$ with
$\lim\limits_{\epsilon\to 0}\error_\epsilon$,
such that for all $\epsilon>0$ sufficiently small, the sets
	\[ D^\kappa_\epsilon = \bigg\{ x\in\domain : \positivefirstorderflow{\vortex_\epsilon}(x) \geq \int_\domain\positivefirstorderflow{\vortex_\epsilon}\dif\positivepart{\vortex_\epsilon} - \kappa\frac{\sup_\domain\depth}{4\pi}\vortexstrength_\epsilon\log\frac{1}{\epsilon} \bigg\} \]
and
	\[ U^\kappa_\epsilon = \bigg\{ x\in\domain : \negativefirstorderflow{\vortex_\epsilon}(x) \geq \int_\domain\negativefirstorderflow{\vortex_\epsilon}\dif\negativepart{\vortex_\epsilon} - \kappa\frac{\sup_\domain\depth}{4\pi}\vortexstrength_\epsilon\log\frac{1}{\epsilon} \bigg\} ,\]
have diameter smaller than $\epsilon^\varsigma$, and
	\[ \limsup_{\epsilon\to 0}\int_{\domain\setminus D^\kappa_\epsilon}\positivepart{\vortex_\epsilon}\mumeasure
	\leq \frac{1}{\kappa}\frac{4\pi}{\sup_\domain\depth}\vortexstrength_\epsilon \bigg( \error_\epsilon + \frac{C}{\log\frac{1}{\epsilon}} \bigg) \]
and
	\[ \limsup_{\epsilon\to 0}\int_{\domain\setminus U^\kappa_\epsilon}\negativepart{\vortex_\epsilon}\mumeasure
	\leq \frac{1}{\kappa}\frac{4\pi}{\sup_\domain\depth}\vortexstrength_\epsilon \bigg( \error_\epsilon + \frac{C}{\log\frac{1}{\epsilon}} \bigg) .\]
Furthermore, given $X_\epsilon\in D^\kappa_\epsilon$ and $Y_\epsilon\in U^\kappa_\epsilon$ for all
$\epsilon>0$ sufficiently small, we have
	\[ \lim_{\epsilon\to 0}\depth(X_\epsilon)
		= \sup_\domain\depth = \lim_{\epsilon\to 0}\depth(Y_\epsilon) .\]
\end{theorem}
\endgroup
\setcounter{theorem}{0} %assign desired value to theorem counter
For any positive measurable function $f$, the notation $\dif f$ used in \cref{weakConcentration} indicates the 
set function defined for all measurable set $A\subseteq\domain$ as
	\[ \dif f(A) := \frac{1}{\norm{f}{L^1_{\mu}}}\int_Af\mumeasure \]
if $f$ is non vanishing; and identically null otherwise.

\subsection{Essential concentration result}
\begin{proposition}\label{strongConcentration}
There exists $\varsigma>0$ such that, for all sufficiently small $\epsilon>0$, and for all $X_\epsilon\in\{\positivepart{\vortex_\epsilon}>0\}$,
for all $Y_\epsilon\in\{\negativepart{\vortex_\epsilon}>0\}$, we have
	\[ \int_{\domain\setminus B(X_\epsilon,\epsilon^\varsigma)}\dif\positivepart{\vortex_\epsilon} = 0 \]
and
	\[ \int_{\domain\setminus B(Y_\epsilon,\epsilon^\varsigma)}\dif\negativepart{\vortex_\epsilon} = 0 .\]
\end{proposition}
The proof is an adaptation of techniques due to Turkington~\citelist{\cite{TurkingtonSteady1}\cite{TurkingtonSteady2}}
(see also~Elcrat~\&~Miller~\cite{ElcratMiller} for a similar problem). Our situation is very similar
but the computations turn out to be a bit more involved
due to the presence of islands in the lake ($\domain$ may not be simply connected) and because the vortex changes sign.
\begin{proof}
Fix $\epsilon>0$. We prove the claim for the positive part
$\positivepart{\vortex_\epsilon}$ only. The claim for the negative part
follows by symmetry.
According to \cref{weakConcentration}, there exists $\kappa>0$
and $\varsigma>0$ such that for all sufficiently small $\epsilon>0$, the diameter of the set
	\[ D^\kappa_\epsilon = \bigg\{ x\in\domain : \positivefirstorderflow{\vortex_\epsilon}(x) \geq \int_\domain\positivefirstorderflow{\vortex_\epsilon}\dif\positivepart{\vortex_\epsilon} - \kappa\frac{\sup_\domain\depth}{4\pi}\vortexstrength_\epsilon\log\frac{1}{\epsilon} \bigg\} \]
is smaller than $\epsilon^\varsigma$. Furthermore, we have
	\[ \int_{\domain\setminus D^\kappa_\epsilon}\positivepart{\vortex_\epsilon}\mumeasure
		\leq \frac{1}{\kappa}\frac{4\pi}{\sup_\domain\depth}\vortexstrength_\epsilon \bigg( \error_\epsilon + \frac{C}{\log\frac{1}{\epsilon}} \bigg) ,\]
for some family of numbers $\{\error_\epsilon:\epsilon>0\}$
with $\lim\limits_{\epsilon\to 0}\error_\epsilon = 0$.
We are going to prove that the latter estimate may be improved as
		\[ \int_{\domain\setminus D^\kappa_\epsilon}\positivepart{\vortex_\epsilon}\mumeasure = 0 .\]
Let us first decompose the stream function $\psi$ associated with $\vortex_\epsilon$
through the elliptic problem~\eqref{ellipticproblem}, as the sum
	\[ \psi = \vortexoperator\vortex_\epsilon + \rectifycirculation_\epsilon\vortex_\epsilon,
		\qquad \rectifycirculation_\epsilon\vortex_\epsilon = \sum^\nbislands_{i=0}\alpha^\epsilon_i(\staticflow_i-\circulation_i) .\]
Since $\vortex_\epsilon$ is a maximizer of the strictly convex functional $\energy_\epsilon$ over $\rearrangement(\vortex_\epsilon)$,
it is also a maximizer over the weak closure
$\weakclosure{\rearrangement(\vortex_\epsilon)}$, by
weak continuity of $\energy_\epsilon$~\cite{myArticle}*{Section~2}. Since the latter compact set
is also convex in $L^p(\domain,\mu)$~\cite{BurtonRearrangementOfFunctions}*{Theorem~6},
it is straightforward to check that $\vortex_\epsilon$ is the only maximizer of the linear functional
	\[ L : L^p(\domain,\mu)\to \reals: L(\vortex) = \int_\domain (\vortexoperator+\rectifycirculation_\epsilon)(\vortex_\epsilon)\,\vortex\mumeasure \]
over the set $\rearrangement(\vortex_\epsilon)$. Indeed, observe that $(\vortexoperator+\rectifycirculation_\epsilon)(\vortex_\epsilon)$
belongs to the subgradient of $\energy_\epsilon$ at point $\vortex_\epsilon$, so that by definition of subgradient at $\vortex_\epsilon$
we have, for all $\vortex\in L^p(\domain,\mu)$:
	\[ \energy_\epsilon(\vortex)-\energy_\epsilon(\vortex_\epsilon)
		\geq L(\vortex - \vortex_\epsilon) .\]
Taking $\vortex\in\rearrangement(\vortex_\epsilon)$ and $\lambda\in(0,1)$, we define $\vortex_\lambda=\lambda\vortex+(1-\lambda)\vortex_\epsilon$.
Hence we have
	\[ \energy_\epsilon(\vortex_\lambda) - \energy_\epsilon(\vortex_\epsilon)
		\geq L(\vortex_\lambda - \vortex_\epsilon) .\]
By linearity of $L$ and strict convexity of $\energy_\epsilon$, we have
	\[ \lambda\energy_\epsilon(\vortex) + (1-\lambda)\energy_\epsilon(\vortex_\epsilon) - \energy_\epsilon(\vortex_\epsilon)
		> \lambda L(\vortex) + (1-\lambda)L(\vortex_\epsilon) -L(\vortex_\epsilon) ,\]
Since the weak closure $\weakclosure{\rearrangement(\vortex_\epsilon)}$ is convex~\cite{BurtonRearrangementOfFunctions}*{Theorem~6},
we obtain
	\[ 0\geq \energy_\epsilon(\vortex) - \energy_\epsilon(\vortex_\epsilon) > L(\vortex) - L(\vortex_\epsilon) ,\]
and thus $\vortex_\epsilon$ strictly maximizes $L$ over its set of $\mu$-rearrangements $\rearrangement(\vortex_\epsilon)$.
In particular~\cite{BurtonRearrangementOfFunctions}*{Theorem~5},
there exists $\gamma_\epsilon^+>0$ such that
	\[ \big\{ \vortexoperator(\vortex_\epsilon)+\rectifycirculation_\epsilon(\vortex_\epsilon) > \gamma_\epsilon^+ \big\}
		\subseteq \big\{ \vortex_\epsilon>0 \big\} \subseteq \big\{ \vortexoperator(\vortex_\epsilon)+\rectifycirculation_\epsilon(\vortex_\epsilon) \geq \gamma_\epsilon^+ \big\} .\]
Let $\phi=\positivepart{\vortexoperator(\vortex_\epsilon)+\rectifycirculation_\epsilon(\vortex_\epsilon)-\gamma_\epsilon^+}$.
It is direct to prove that
	\[ \phi = \phi^\sharp + \sum^\nbislands_{i=0}\positivepart{\beta_{\epsilon;i}-\gamma_\epsilon^+}\staticflow_i ,\]
for some $\phi^\sharp\in\functionspace$ and for
	\[ \beta_{\epsilon;i}
		= \alpha_{\epsilon;i} - \sum^\nbislands_{j=0}\alpha_{\epsilon;j}\circulation_j .\]
Now we observe that
	\begin{align*}
		\int_\domain\scalarproduct{\gradient\phi}{\gradient\phi}{\plane}\ \frac{\dif x}{\depth}
			&= \int_\domain\scalarproduct{\gradient\phi}{\gradient\big(\vortexoperator(\vortex_\epsilon)+\rectifycirculation_\epsilon(\vortex_\epsilon)\big)}{\plane}\ \frac{\dif x}{\depth}
			\\&= \int_\domain\scalarproduct{\gradient\phi^\sharp}{\gradient\vortexoperator(\vortex_\epsilon)}{\plane}\ \frac{\dif x}{\depth}
				+ \sum^\nbislands_{j=0}\positivepart{\beta_{\epsilon;j}-\gamma_\epsilon^+}
					\int_\domain\scalarproduct{\gradient\rectifycirculation_\epsilon(\vortex_\epsilon)}{\gradient\staticflow_j}{\plane}\ \frac{\dif x}{\depth}
			\\&= \int_\domain\phi^\sharp\vortex_\epsilon\mumeasure
				+ \sum^\nbislands_{j=0}\positivepart{\beta_{\epsilon;j}-\gamma_\epsilon^+}
					\int_\domain\scalarproduct{\gradient\rectifycirculation_\epsilon(\vortex_\epsilon)}{\gradient\staticflow_j}{\plane}\ \frac{\dif x}{\depth} .
	\end{align*}
By construction of $\alpha_{\epsilon;0},\dots,\alpha_{\epsilon,\nbislands}$, we have for all $j\in\{0,\dots,\nbislands\}$:
	\begin{equation}\label{memo1}
	\int_\domain\scalarproduct{\gradient\phi}{\gradient\phi}{\plane}\ \frac{\dif x}{\depth}
		= \int_\domain\phi\vortex_\epsilon\mumeasure -  \sum^\nbislands_{j=0}\positivepart{\beta_{\epsilon;j}-\gamma_\epsilon^+}\circulation_j (2\tau-1)\vortexstrength_\epsilon .
	\end{equation}
This may also be written as
	\begin{multline*}
		\int_\domain\scalarproduct{\gradient\phi}{\gradient\phi}{\plane}\ \frac{\dif x}{\depth}
		= \int_\domain\big( \phi  - \positivepart{\beta_{\epsilon;0}-\gamma_\epsilon^+} \big)\vortex_\epsilon\mumeasure
			\\+ (2\tau-1)\vortexstrength_\epsilon\bigg(
				\positivepart{\beta_{\epsilon;0}-\gamma_\epsilon^+} - \sum^\nbislands_{j=0}\positivepart{\beta_{\epsilon;j}-\gamma_\epsilon^+}\circulation_j
			\bigg) .
	\end{multline*}
Observe that $u=\phi-\positivepart{\beta_{\epsilon;0}-\gamma_\epsilon^+}$ belongs to $W^{1,2}_0(\domain_0)$.
In particular, if $p\in(1,2)$, one may choose $q>2$ such that
$q^\star=2q(2-q)=p'=p/(p-1)$, so that by Sobolev's inequality in
$W^{1,q}_0(\domain_0)$, we obtain
	\[ \norm{u}{L^{p'}(\domain)} \leq \norm{u}{L^{p'}(\domain_0)}
		\leq \ncste\norm{\gradient u}{L^q(\domain_0)} = \cste\norm{\gradient u}{L^q(\domain)}
		\leq \ncste \mu\big(\{u>0\}\big)^{\frac{1}{p'}}
			\norm{\gradient u}{\functionspace} , \]
where in the last step we have used Hölder's inequality in
$L^{2/q}(\domain,\depth^{-1}\lebesguemeasure)$. The same estimate holds whenever $p\geq 2$,
because the control condition~\eqref{criterionconvergence} holds with $p$ replaced by $q\in(1,p]$,
by Hölder's inequality. From this we conclude that
	\[ \int_\domain u\vortex_\epsilon\mumeasure
		\leq \ncste\vortexstrength_\epsilon\ \norm{u}{\functionspace} .\]
Hence we infer the estimate
	\[ \int_\domain\scalarproduct{\gradient\phi}{\gradient\phi}{\plane}\ \frac{\dif x}{\depth}
		\leq C\vortexstrength_\epsilon \Bigg(\int_\domain\scalarproduct{\gradient\phi}{\gradient\phi}{\plane}\ \frac{\dif x}{\depth}\Bigg)^{\frac{1}{2}}
			+ C\vortexstrength_\epsilon \bigg(
				\positivepart{\beta_{\epsilon;0}-\gamma_\epsilon^+} - \sum^\nbislands_{j=0}\positivepart{\beta_{\epsilon;j}-\gamma_\epsilon^+}\circulation_j
			\bigg) ,\]
from which we deduce that
	\begin{equation}\label{eq1.strongConcentration}
	\Bigg(\int_\domain\scalarproduct{\gradient\phi}{\gradient\phi}{\plane}\ \frac{\dif x}{\depth}\Bigg)^{\frac{1}{2}}
		\leq \frac{C\vortexstrength_\epsilon + \sqrt{C^2\vortexstrength_\epsilon^2 + 4C\vortexstrength_\epsilon\bigg(
				\positivepart{\beta_{\epsilon;0}-\gamma_\epsilon^+} - \sum^\nbislands_{j=0}\positivepart{\beta_{\epsilon;j}-\gamma_\epsilon^+}\circulation_j
			\bigg) }}{2} .
	\end{equation}
Although $\gamma_\epsilon^+$ may be of great order in comparison with $\vortexstrength_\epsilon$,
the particular structure of the above estimate allows us to write, for all $\epsilon>0$ sufficiently small:
	\[ \Bigg(\int_\domain\scalarproduct{\gradient\phi}{\gradient\phi}{\plane}\ \frac{\dif x}{\depth}\Bigg)^{\frac{1}{2}}
		\leq \ncste\vortexstrength_{\epsilon}. \]
Injecting this estimate in equation~\eqref{memo1}, and using the definition
of $\phi$, we obtain
	\[ \ncste\vortexstrength_\epsilon^2
		\geq \int_\domain\big(\vortexoperator\vortex_\epsilon + \rectifycirculation_\epsilon\vortex_\epsilon\big)\positivepart{\vortex_\epsilon}\mumeasure
			- \gamma^+_\epsilon\int_\domain\positivepart{\vortex_\epsilon}\mumeasure
			- \sum^\nbislands_{j=0}\positivepart{\beta_{\epsilon;j}-\gamma_\epsilon^+}\circulation_j \int_\domain\vortex_\epsilon\mumeasure \geq 0 . \]
According to \cref{upperBoundFirstOrder}, we have
	\[ \gamma^+_\epsilon\tau\vortexstrength_\epsilon
			+ (2\tau-1)\vortexstrength_\epsilon\sum^\nbislands_{j=0}\positivepart{\beta_{\epsilon;j}-\gamma_\epsilon^+}\circulation_j
		\geq \int_\domain\positivepart{\vortex_\epsilon}\positivefirstorderflow{\vortex_\epsilon}\mumeasure
		- \ncste\vortexstrength_\epsilon^2\bigg(\frac{1}{\log\frac{1}{\epsilon}}+\error_\epsilon\bigg)\log\frac{1}{\epsilon}. \]
As $\epsilon\to 0$, the right hand side blows up as $\vortexstrength_\epsilon^2\log\frac{1}{\epsilon}$
by the concentration result in \cref{weakConcentration}.
This forces to have $\gamma_\epsilon^+\geq\max\limits_{0\leq i\leq\nbislands}\beta_{\epsilon;i}$ for all $\epsilon>0$ sufficiently small.
For all sufficiently small $\epsilon>0$, we thus have
	\[ \gamma^+_\epsilon
		\geq \int_\domain\positivefirstorderflow{\vortex_\epsilon}\dif\positivepart{\vortex_\epsilon}
		- \ncste\vortexstrength_\epsilon\bigg(\frac{1}{\log\frac{1}{\epsilon}}+\error_\epsilon\bigg)\log\frac{1}{\epsilon}. \]
For all $x\in\{\vortex_\epsilon>0\}$, we obtain by definition of $\gamma_\epsilon^+$
	\[ \vortexoperator\vortex_\epsilon(x) + \rectifycirculation_\epsilon\vortex_\epsilon(x)
		\geq \gamma_\epsilon^+
		\geq \int_\domain\positivefirstorderflow{\vortex_\epsilon}\dif\positivepart{\vortex_\epsilon}
		- \ncste\vortexstrength_\epsilon\bigg(\frac{1}{\log\frac{1}{\epsilon}}+\error_\epsilon\bigg)\log\frac{1}{\epsilon} .\]
Using the fact that the Green's function $\greenlaplace$ is a positive function, and the uniform bounds on $\rectifykernel$ and the flows $\staticflow_0,\dots,\staticflow_{\nbislands}$,
we conclude the lower estimate, for all sufficiently small $\epsilon>0$ and for all $x\in\{\vortex_\epsilon>0\}$:
	\[ \positivefirstorderflow{\vortex_\epsilon}(x)\geq \int_\domain\positivefirstorderflow{\vortex_\epsilon}\dif\positivepart{\vortex_\epsilon}
		- \kappa\vortexstrength_\epsilon\log\frac{1}{\epsilon} .\]
In particular, we have
	\[ \big\{\positivepart{\vortex_\epsilon}>0\big\} = \big\{\vortex_\epsilon>0\big\} \subseteq \big\{\vortexoperator\vortex_\epsilon+\rectifycirculation_\epsilon\vortex_\epsilon\geq\gamma^+_\epsilon\big\} \subseteq D^\kappa_\epsilon ,\]
which proves the claim for the positive part of the vortex.
\end{proof}

\section{Repulsion effects}

\subsection{Riesz-Sobolev rearrangement inequality}
A important feature in the theory of standard symmetrization is the use of radial competitors
together with geometric inequalities. In this direction, we are going to prove a variant of the well-known
Riesz-Sobolev rearrangement inequality. One should pay attention that we do not work with the Lebesgue measure,
but with the weighted measure $\mu(x)=\depth(x)\dif x$. This measure may not behave nicely
with respect to geometric transformations. Rather than assuming geometric conditions on $\depth$,
we propose an \emph{asymptotic variant} of the Riesz-Sobolev inequality. This will be sufficient for our purposes.

We first require the following standard lemma:
\begin{proposition}\label{symmetrizearoundpoint}
For all $x\in\domain$, there exists a function $\symmetrizearoundpoint{x}{\cdot}:L^1(\domain,\mu)\to L^1(\domain,\mu)$
such that for all positive function $\vortex\in L^1(\domain,\mu)$ we have
	\begin{enumerate}
		\item $\symmetrizearoundpoint{x}{\vortex}\in\rearrangement(\vortex)$;
		\item the superlevel sets of $\symmetrizearoundpoint{x}{\vortex}$ are balls (in $\domain$) centered on $x$.
	\end{enumerate}
Furthermore, for all $x\in\domain$ and for all positive functions $\vortex_1,\vortex_2\in L^1(\domain,\mu)$, we have
	\[ \norm{\symmetrizearoundpoint{x}{\vortex_1} - \symmetrizearoundpoint{x}{\vortex_2}}{L^1_{\mu}}
		\leq \norm{\vortex_1-\vortex_2}{L^1_{\mu}} .\]
\end{proposition}
\begin{proof}
The first part of the claim was proved in~\cite{myArticle}*{Proposition~2.1} and it is a standard construction
in the field of symmetrizations.
For the second claim, let $x\in\domain$ and $\vortex_1,\vortex_2\in L^1(\domain,\mu)$.
Fix $\lambda\in\reals^+$. By construction, the sets $\{\symmetrizearoundpoint{x}{\vortex_1}\geq\lambda\}$ and $\{\symmetrizearoundpoint{x}{\vortex_2}\geq\lambda\}$
are balls centered on $x$. Without loss of generality, assume that
	\[ \mu\big(\{\symmetrizearoundpoint{x}{\vortex_1}\geq\lambda\}\big)
		\leq \mu\big(\{\symmetrizearoundpoint{x}{\vortex_2}\geq\lambda\}\big) ,\]
so that $\{\symmetrizearoundpoint{x}{\vortex_1}\geq\lambda\}\subseteq\{\symmetrizearoundpoint{x}{\vortex_2}\geq\lambda\}$.
One then estimates through a direct computation
	\[ \mu\big( \{\symmetrizearoundpoint{x}{\vortex_1}\geq\lambda\}\Delta \{\symmetrizearoundpoint{x}{\vortex_2}\geq\lambda\} \big)
		\leq \mu\big(\{\vortex_1\geq\lambda\}\Delta\{\vortex_2\geq\lambda\} \big) .\]
The conclusion for the $L^1(\domain,\mu)$-norms follows from Cavalieri's principle.
\end{proof}

\begin{proposition}[Asymptotic Riesz-Sobolev rearrangement inequality]\label{RieszSobolev}
Let $R>0$ and for all $\epsilon>0$, let $X^\star_\epsilon\in\domain$
be such that
$\mu\big(B(X^\star_\epsilon,R)\big)\geq \epsilon^2$
with $\inf\limits_{B(X^\star_\epsilon,R)}\depth>0$.
There exists $C>0$ such that, for all sufficiently small
$\epsilon>0$ and for all $X\in\{\positivepart{\vortex_\epsilon}>0\}$, we have
	\begin{multline*}
		\iint\limits_{\domain\times\domain}\log\frac{\diameter(\domain)}{|x-y|}
			\dif\positivepart{\vortex_\epsilon}
			\dif\positivepart{\vortex_\epsilon}(x,y)
			\leq \iint\limits_{\domain\times\domain}\log\frac{\diameter(\domain)}{|x-y|}
				\dif\symmetrizearoundpoint{X^\star_\epsilon}{\positivepart{\vortex_\epsilon}}
				\dif\symmetrizearoundpoint{X^\star_\epsilon}{\positivepart{\vortex_\epsilon}}(x,y)
		\\ +\log\sqrt{\frac{\depth(X)}{\depth(X^\star_\epsilon)}}+ C\Big( \omega_{\depth}\big(\diameter(\{\symmetrizearoundpoint{X^\star_\epsilon}{\positivepart{\vortex_\epsilon}}>0\})\big)
			+ \omega_{\depth}\big(\diameter(\{\positivepart{\vortex_\epsilon}>0\})\big) \Big) .
	\end{multline*}
\end{proposition}
\begin{proof}
Let us fix $\epsilon>0$, and write for short $X^\star=X^\star_\epsilon$, $\vortex=\positivepart{\vortex_\epsilon}$
and $\vortex^\star=\symmetrizearoundpoint{X^\star}{\positivepart{\vortex_\epsilon}}$,
$\theta=\diameter\big(\{\vortex>0\}\big)$ and $\theta^\star=\diameter\big(\{\vortex^\star>0\}\big)$.
Assuming $\epsilon>0$ sufficiently small, we may assume that $\diameter(\{\vortex>0\})\leq R$ and
	\[ \inf_{\{\vortex>0\}}\depth \geq \delta\,\sup_\domain\depth ,
		\qquad \inf_{\{\vortex^\star>0\}}\depth\geq \delta^\star\,\sup_\domain\depth ,\]
for $\delta,\delta^\star>0$. Let $X\in\{\positivepart{\vortex_\epsilon}>0\}$.
Let us write for short $r=\displaystyle\sqrt{\frac{\depth(X)}{\depth(X^\star)}}$ and
	\[ \phi : \plane\to\plane : \phi(x) = X^\star + r(x-X) .\]
Finally define the following auxiliary functions: $\vortex^\Delta$ is the \emph{Lebesgue} nonincreasing Lebesgue re\-ar\-ran\-ge\-ment
of $\vortex$ around the point $X$, and $\xi=\vortex^\Delta\circ\phi^{-1}$.
From the definition it follows that every super level set of $\xi$ is a ball centered on $X^\star$,
with
	\[ \lebesguemeasure\big( \{\xi\geq t\} \big) = r^2\lebesguemeasure\big( \{\vortex\geq t\} \big) .\]
From this estimate we have
	\[ \Big| \depth(X^\star)\lebesguemeasure\big(\{\xi\geq t\}\big) - \mu\big(\{\vortex\geq t\} \big) \Big|
		\leq \frac{\omega_\depth(\theta)}{\delta\sup_\domain\depth}\mu\big(\{\vortex\geq t\}\big)  .\]
Similarly, one may also compute
	\[ \Big| \depth(X^\star)\lebesguemeasure\big(\{\vortex^\star\geq t\}\big) - \mu\big(\{\vortex\geq t\}\big) \Big|
		\leq \frac{\omega_\depth(\theta^\star)}{\delta^\star\sup_\domain\depth}\mu\big(\{\vortex\geq t\}\big) .\]
According to the standard Riesz-Sobolev rearrangement inequality~\citelist{\cite{LiebLoss}\cite{CroweZweibelRosenbloom}}
and using the change of variable formula, we have
	\begin{multline*}
		\iint\limits_{\plane\times\plane}\log\frac{r^{-1}\theta^\star}{|x-y|}\vortex(x)\vortex(y)\dif (x,y)
			\leq \iint\limits_{\plane\times\plane}\positivepart{\log\frac{r^{-1}\theta^\star}{|x-y|}}\vortex^\Delta(x)\vortex^\Delta(y)\dif (x,y) 
		\\ \leq \frac{1}{r^4}\iint\limits_{\plane\times\plane}\positivepart{\log\frac{\theta^\star}{|x-y|}}\xi(x)\xi(y)\dif (x,y) .
	\end{multline*}
This yields to
	\begin{equation}\label{eqnref1}
		\depth(X)^2\iint\limits_{\plane\times\plane}\log\frac{r^{-1}\theta^\star}{|x-y|}\vortex(x)\vortex(y)\dif (x,y)
			\leq \depth(X^\star)^2\iint\limits_{\plane\times\plane}\positivepart{\log\frac{\theta^\star}{|x-y|}}\xi(x)\xi(y)\dif(x,y) .
	\end{equation}
On the other hand, we also have
	\[ \iint\limits_{\plane\times\plane}\positivepart{\log\frac{\theta^\star}{|x-y|}}\vortex^\star(x)\vortex^\star(y)\mumeasure\mumeasure(x,y)
		= \iint\limits_{\plane\times\plane}\log\frac{\theta^\star}{|x-y|}\vortex^\star(x)\vortex^\star(y)\mumeasure\mumeasure(x,y) ,\]
and
	\begin{multline*}
		\depth(X^\star)^2\iint\limits_{\plane\times\plane}\positivepart{\log\frac{\theta^\star}{|x-y|}}\xi(x)\xi(y)\dif(x,y)
			- \iint\limits_{\plane\times\plane}\positivepart{\log\frac{\theta^\star}{|x-y|}}\vortex^\star(x)\vortex^\star(y)\mumeasure\mumeasure(x,y)
	\\= \int_{\plane}\big(\depth(X^\star)\xi(x)-\depth(x)\vortex^\star(x)\big)\int_{\plane}\positivepart{\log\frac{\theta^\star}{|x-y|}}
		\big(\depth(X^\star)\xi(y)+\depth(y)\vortex^\star(y)\big)\dif(x,y) .
	\end{multline*}
Let us estimate
	\begin{multline*}
		\int\limits_{\plane}\big|\depth(X^\star)\xi(x)-\depth(x)\vortex^\star(x)\big|\dif x
			\\\leq \int\limits_{\plane}\depth(X^\star)\big|\vortex^\star(x)-\xi(x)\big|\dif x
				+ \int\limits_{\plane}\big|\depth(X^\star)-\depth(x)\big|\,\vortex^\star(x)\dif x
			\\\leq \depth(X^\star)\int_0^{+\infty}\lebesguemeasure\big( \{\vortex^\star\geq t\}\Delta\{\xi\geq t\}\big) \dif t
				+ \omega_{\depth}(\theta^\star)\int\limits_{\plane}\vortex^\star(x)\dif x .
	\end{multline*}
For the second term, we have
	\[ \int\limits_{\plane}\vortex^\star(x)\dif x
		\leq \frac{1}{\delta^\star\sup_\domain\depth}\,\int\limits_{\plane}\vortex^\star(x)\mumeasure(x) = \frac{\tau\vortexstrength_\epsilon}{\delta^\star\sup_\domain\depth} .\]
For the first term, we have
	\begin{align*}
		\depth(X^\star)\lebesguemeasure\big( \{\vortex^\star\geq t\}\Delta\{\xi\geq t\}\big)
			&= \Big| \depth(X^\star)\lebesguemeasure\big( \{\vortex^\star\geq t\} \big) - \depth(X^\star) \lebesguemeasure\big( \{\xi\geq t\}\big) \Big|
			\\&\leq \bigg( \frac{\omega_\depth(\theta^\star)}{\delta^\star\sup_\domain\depth}  +\frac{\omega_\depth(\theta)}{\delta\sup_\domain\depth}
			\bigg)\ \mu\big(\{\vortex\geq t\}\big)
	\end{align*}
Integrating over $t\in [0,+\infty)$ yields
	\[ \depth(X^\star)\int_0^{+\infty}\lebesguemeasure\big( \{\vortex^\star\geq t\}\Delta\{\xi\geq t\}\big) \dif t
		\leq \bigg( \frac{\omega_\depth(\theta^\star)}{\delta^\star\sup_\domain\depth}  +\frac{\omega_\depth(\theta)}{\delta\sup_\domain\depth}
			\bigg)\tau \vortexstrength_\epsilon .\]
Therefore we have
	\[ \int\limits_{\plane}\big|\depth(X^\star)\xi(x)-\depth(x)\vortex^\star(x)\big|\dif x
		\leq \bigg( \frac{\omega_\depth(\theta^\star)}{\delta^\star\sup_\domain\depth}  +\frac{\omega_\depth(\theta)}{\delta\sup_\domain\depth}
			\bigg)\tau \vortexstrength_\epsilon .\]
Using the fact that $\theta^\star$ is of order $\epsilon$,
we obtain some constant $\ncste(\delta,\delta^\star,\depth)>0$ such that
	\begin{multline*}
		\depth(X^\star)^2\iint\limits_{\plane\times\plane}\positivepart{\log\frac{\theta^\star}{|x-y|}}\xi(x)\xi(y)\dif(x,y)
			- \iint\limits_{\plane\times\plane}\log\frac{\theta^\star}{|x-y|}\vortex^\star(x)\vortex^\star(y)\mumeasure\mumeasure(x,y)
		\\ \leq \cste\Big( \omega_{\depth}(\theta^\star) + \omega_{\depth}(\theta) \Big)\vortexstrength_\epsilon^2 .
	\end{multline*}
Injecting the previous estimate in equation~\eqref{eqnref1} yields to
	\begin{multline*}
		\depth(X)^2\iint\limits_{\plane\times\plane}\positivepart{\log\frac{r^{-1}\theta^\star}{|x-y|}}\vortex(x)\vortex(y)\dif(x,y)
			- \iint\limits_{\plane\times\plane}\log\frac{\theta^\star}{|x-y|}\vortex^\star(x)\vortex^\star(y)\mumeasure\mumeasure(x,y)
		\\ \leq \cste\Big( \omega_{\depth}(\theta^\star) + \omega_{\depth}(\theta) \Big)\vortexstrength_\epsilon^2 .
	\end{multline*}
Since $\depth$ is Hölder continuous and $\vortex_\epsilon$ concentrates on points of maximal depth,
this may be rewritten as
	\begin{multline*}
		\iint\limits_{\plane\times\plane}\log\frac{r^{-1}\theta^\star}{|x-y|}\vortex(x)\vortex(y)\mumeasure\mumeasure(x,y)
			- \iint\limits_{\plane\times\plane}\log\frac{\theta^\star}{|x-y|}\vortex^\star(x)\vortex^\star(y)\mumeasure\mumeasure(x,y)
		\\ \leq \ncste\Big( \omega_{\depth}(\theta^\star) + \omega_{\depth}(\theta) \Big)\vortexstrength_\epsilon^2 ,
	\end{multline*}
and therefore
	\begin{multline*}
		\iint\limits_{\plane\times\plane}\log\frac{\diameter(\domain)}{|x-y|}\vortex(x)\vortex(y)\mumeasure\mumeasure(x,y)
			\leq \iint\limits_{\plane\times\plane}\log\frac{\domain(\domain)}{|x-y|}\vortex^\star(x)\vortex^\star(y)\mumeasure\mumeasure(x,y)
		\\ +\tau^2\vortexstrength_\epsilon^2\log\sqrt{\frac{\depth(X)}{\depth(X^\star)}}+ \cste\Big( \omega_{\depth}(\theta^\star) + \omega_{\depth}(\theta) \Big)\vortexstrength_\epsilon^2 .
		\qedhere
	\end{multline*}
\end{proof}

\subsection{A priori estimate for repulsion}

The aim of this section is to prove the following a priori estimates for the repulsion of the pair.
\begin{proposition}\label{repulsionLog}
Let $X\in\closure{\domain}$ be a maximizer of $\depth$,
and assume that $\domain$ satisfies an interior cone condition at $X$.
There exists constants $C_1,C_2,C_3>0$ and exponents $\gamma_1,\gamma_2,\gamma_3>0$
such that, for all sufficiently small $\epsilon>0$, we have
	\begin{align*}
		\distance\big(\{\positivepart{\vortex_\epsilon}>0\},\boundary\domain\big)
			&\geq C_1\bigg(\frac{1}{\log\frac{1}{\epsilon}}\bigg)^{\gamma_1},
	\\	\distance\big(\{\negativepart{\vortex_\epsilon}>0\},\boundary\domain\big)
			&\geq C_2\bigg(\frac{1}{\log\frac{1}{\epsilon}}\bigg)^{\gamma_2},
	\\	\distance\big(\{\positivepart{\vortex_\epsilon}>0\},\{\negativepart{\vortex_\epsilon}>0\}\big)
			&\geq C_3\bigg(\frac{1}{\log\frac{1}{\epsilon}}\bigg)^{\gamma_3}.
	\end{align*}
\end{proposition}
The cone condition we impose may be relaxed by a more general cusp-like condition, but
we think that this improvement does not bring a better understanding of the general behavior.
\begin{lemma}[\cite{TurkingtonFriedman}*{Lemma~2.2}]\label{regularestimate}
Let $\mathcal{U}=\text{\normalfont interior}\big(\closure{\domain}\big)$.
For all $x,y\in\domain$, we have
	\begin{multline*}
		\frac{1}{2\pi}\log\frac{\diameter(\domain)}{\max\big\{|x-y|,\distance(x,\partial\domain),\distance(y,\partial\domain)\}}
		\\\geq \regulargreenlaplace(x,y)
		\geq \frac{1}{2\pi}\log\frac{\diameter(\domain)}{|x-y|+2\max\big\{
		\distance(x,\boundary\mathcal{U}),
		\distance(y,\boundary\mathcal{U}) \big\}} .
	\end{multline*}
\end{lemma}
\begin{proof}
The upper bound for $\regulargreenlaplace$ follows from the fact
that the Green's function $\greenlaplace$ is positive,
and from the weak maximum principle for $\regulargreenlaplace$.
Since $\regulargreenlaplace$ is a symmetric function (because so is
the Green's function $\greenlaplace$), we may assume without loss
of generality that
	\[  \distance\big( x , \boundary\mathcal{U} \big)
		\leq \distance(y , \boundary\mathcal{U} \big) .\]
Let $B(z,r)$ be a ball in $\plane\setminus\closure{\domain}$,
and consider the function
	\[ \tilde{\greenlaplace} : \plane\setminus\bigg\{y, z+\frac{r^2(y-z)}{|y-z|^2} \bigg\}\to\reals : \tilde{\greenlaplace}(w)
		= \frac{1}{2\pi}\log\frac{\Big|\frac{|y-z|}{r}(w-z) - \frac{r(y-z)}{|y-z|}\Big|}{|w-y|} .\]
This function is harmonic on $\plane\setminus\big\{y, z+\frac{r^2(y-z)}{|y-z|^2} \big\}$. It is also null on $\boundary B(z,r)$,
and it is nonnegative on $\boundary\domain$.
To see this, one may compute, for all $w\in\plane\setminus\big\{y, z+\frac{r^2(y-z)}{|y-z|^2} \big\}$:
	\[ 4\pi\tilde{\greenlaplace}(w)
		= \log\frac{\frac{|y-z|^2}{r^2}|w-z|^2 + r^2 - 2(w-z)\cdot(y-z)}{|w-z|^2 + |y-z|^2 - 2(w-z)\cdot(y-z)} ,\]
which vanishes if $|w-z|=r$; and on $\boundary\domain$ we have  $|w-z|\geq r$ and thus
$\tilde{g}(w)\geq 0$. It follows from the weak maximum principle that for all $w\in\domain$, we have
$\greenlaplace(w,y)\leq\tilde{\greenlaplace}(w)$. In particular we obtain
	\begin{align*}
		\regulargreenlaplace(x,y)
			&\geq \frac{1}{2\pi}\log\frac{\diameter(\domain)}{\Big|\frac{|y-z|}{r}(x-z) - \frac{r(y-z)}{|y-z|}\Big|}
			\\&= \frac{1}{2\pi}\log\frac{\diameter(\domain)\frac{|y-z|}{r}}{\Big|(x-z) - \frac{r^2(y-z)}{|y-z|^2}\Big|}
			\\&= \frac{1}{2\pi}\log\frac{\diameter(\domain)\frac{|y-z|}{r}}{\Big|(x-y) + \frac{|y-z|^2-r^2}{|y-z|^2}(y-z)\Big|}
			\\&\geq \frac{1}{2\pi}\log\frac{\diameter(\domain)\frac{|y-z|}{r}}{|x-y| + \frac{|y-z|+r}{|y-z|}\,\big(|y-z|-r\big)} ,
	\end{align*}
where in the last line we have used the triangular inequality. Now using the fact that $y\notin B(z,r)$, we have
	\[ \regulargreenlaplace(x,y)
		\geq \frac{1}{2\pi}\log\frac{\diameter(\domain)}{|x-y| + 2\big(|y-z|-r\big)} .\]
This inequality is true for all ball $B(z,r)\subseteq\plane\setminus\closure{\domain}$,
so that
	\[ \regulargreenlaplace(x,y)
		\geq \frac{1}{2\pi}\log\frac{\diameter(\domain)}{|x-y| + 2\,\distance\big(y,\boundary\mathcal{U}\big)} .\]
As mentioned above, the conclusion follows by symmetry.
\end{proof}
\begin{proof}[Proof of \cref{repulsionLog}]
For all $\epsilon>0$, let $X_\epsilon\in\{\positivepart{\vortex_\epsilon}>0\}$
and $Y_\epsilon\in\{\negativepart{\vortex_\epsilon}>0\}$.
Since $\depth$ is assumed to be Hölder continuous on $\closure{\domain}$,
let $\alpha>0$ be such that $\depth\in C^{0,\alpha}(\closure{\domain})$.
Let us also define $\mathcal{U}=\text{interior}(\closure{\domain})$,
and consider two families of
points $\{X_\epsilon^\star:\epsilon>0\}$ and $\{Y_\epsilon^\star:\epsilon>0\}$
such that there exists a constant $C>0$ such that for all sufficiently
small $\epsilon>0$, we have
	\[ \frac{1}{C\big(\log\frac{1}{\epsilon}\big)^{\frac{1}{\alpha}}}
		\leq \distance(X_\epsilon^\star,\boundary\mathcal{U})
		\leq \frac{C}{\big(\log\frac{1}{\epsilon}\big)^{\frac{1}{\alpha}}} ,\]
and the same estimates hold for both
$\distance(Y_\epsilon^\star,\boundary\mathcal{U})$
and $\distance(X_\epsilon^\star,Y_\epsilon^\star)$.
Such families exist, because there exists a point $X$ in $\closure{\domain}$
that maximize $\depth$ and $\closure{\domain}$ satisfies an interior cone condition at $X$.
For all sufficiently small $\epsilon>0$, the function
	\[ \tilde{\vortex}_\epsilon
	= \symmetrizearoundpoint{X_\epsilon}{\positivepart{\vortex_\epsilon}}
	-\symmetrizearoundpoint{Y_\epsilon}{\negativepart{\vortex_\epsilon}}
	\]
is a $\mu$-rearrangement of $\vortex_\epsilon$, and in particular we have
$\energy_\epsilon(\tilde{\vortex}_\epsilon)\leq\energy_\epsilon(\vortex_\epsilon)$.
We are now going to estimate the above energies.
Using the integral kernel representation\eqref{integralRepresentation}, page~\pageref{integralRepresentation},
and relying on the boundedness of
$\rectifykernel$ and the boundary flows $\staticflow_0,\dots,\staticflow_\nbislands$,
we first expand
	\begin{multline*}
		\energy_\epsilon(\vortex_\epsilon) \leq
		\int_\domain\positivepart{\vortex_\epsilon}
			\positivefirstorderflow{\vortex_\epsilon}\mumeasure
		+ \int_\domain\negativepart{\vortex_\epsilon}
			\negativefirstorderflow{\vortex_\epsilon}\mumeasure
	\\ - \iint\limits_{\domain\times\domain}
		\positivepart{\vortex_\epsilon}(x)
		\positivepart{\vortex_\epsilon}(y)
		\depth(x)\regulargreenlaplace(x,y)\mumeasure\mumeasure(x,y)
	\\ - \iint\limits_{\domain\times\domain}
		\negativepart{\vortex_\epsilon}(x)
		\negativepart{\vortex_\epsilon}(y)
		\depth(x)\regulargreenlaplace(x,y)\mumeasure\mumeasure(x,y)
	\\- \iint\limits_{\domain\times\domain}
			\positivepart{\vortex_\epsilon}(x)
			\negativepart{\vortex_\epsilon}(y)
		\big(\depth(x)+\depth(y)\big)\greenlaplace(x,y)\mumeasure\mumeasure(x,y)
	+ \ncste\vortexstrength_\epsilon^2 ,
	\end{multline*}
Relying on the Riesz-Sobolev rearrangement inequality, \cref{RieszSobolev},
we also have
	\[ \int_\domain\positivepart{\vortex_\epsilon}
			\positivefirstorderflow{\vortex_\epsilon}\mumeasure
	\leq \int_\domain\frac{\sup_\domain\depth}{\depth}\ \positivepart{\tilde{\vortex}_\epsilon}
			\positivefirstorderflow{\tilde{\vortex}_\epsilon}\mumeasure .\]
By construction of the family
$\{X_\epsilon^\star:\epsilon>0\}$, we have for sufficiently small
$\epsilon>0$ and for all $x\in\{\positivepart{\tilde{\vortex}_\epsilon}>0\}$:
	\[ \Big(\sup_\domain\depth\Big) - \depth(x)
		\leq \ncste \Big( \distance(X^\star_\epsilon,\boundary\mathcal{U}) \Big)^\alpha
		\leq \ncste \frac{1}{\log\frac{1}{\epsilon}} .\]
From this we conclude, for all sufficiently small $\epsilon>0$:
	\[ \int_\domain\positivepart{\vortex_\epsilon}
			\positivefirstorderflow{\vortex_\epsilon}\mumeasure
	\leq \int_\domain\positivepart{\tilde{\vortex}_\epsilon}
			\positivefirstorderflow{\tilde{\vortex}_\epsilon}\mumeasure
	+ \ncste\vortexstrength_\epsilon^2 .\]
Similarly, we have
	\[ \int_\domain\negativepart{\vortex_\epsilon}
			\negativefirstorderflow{\vortex_\epsilon}\mumeasure
	\leq \int_\domain\negativepart{\tilde{\vortex}_\epsilon}
			\negativefirstorderflow{\tilde{\vortex}_\epsilon}\mumeasure
	+ \ncste\vortexstrength_\epsilon^2 .\]
This yields to the estimate
	\begin{multline}\label{energyUpperEstimate}
		\energy_\epsilon(\vortex_\epsilon) \leq
		\int_\domain\positivepart{\tilde{\vortex}_\epsilon}
			\positivefirstorderflow{\tilde{\vortex}_\epsilon}\mumeasure
		+ \int_\domain\negativepart{\tilde{\vortex}_\epsilon}
			\negativefirstorderflow{\tilde{\vortex}_\epsilon}\mumeasure
	\\ - \iint\limits_{\domain\times\domain}
		\positivepart{\vortex_\epsilon}(x)
		\positivepart{\vortex_\epsilon}(y)
		\depth(x)\regulargreenlaplace(x,y)\mumeasure\mumeasure(x,y)
	\\ - \iint\limits_{\domain\times\domain}
		\negativepart{\vortex_\epsilon}(x)
		\negativepart{\vortex_\epsilon}(y)
		\depth(x)\regulargreenlaplace(x,y)\mumeasure\mumeasure(x,y)
	\\- \iint\limits_{\domain\times\domain}
			\positivepart{\vortex_\epsilon}(x)
			\negativepart{\vortex_\epsilon}(y)
		\big(\depth(x)+\depth(y)\big)\greenlaplace(x,y)\mumeasure\mumeasure(x,y)
	+ \ncste\vortexstrength_\epsilon^2 .
	\end{multline}
Now we estimate the energy $\energy_\epsilon(\tilde{\vortex}_\epsilon)$
from below. We first expand $\energy_\epsilon(\tilde{\vortex}_\epsilon)$
using the integral kernel representation\eqref{integralRepresentation}, page~\pageref{integralRepresentation},
	\begin{multline*}
		\energy_\epsilon(\tilde{\vortex}_\epsilon) \geq
		\int_\domain\positivepart{\tilde{\vortex}_\epsilon}
			\positivefirstorderflow{\tilde{\vortex}_\epsilon}\mumeasure
		+ \int_\domain\negativepart{\tilde{\vortex}_\epsilon}
			\negativefirstorderflow{\tilde{\vortex}_\epsilon}\mumeasure
	\\ - \iint\limits_{\domain\times\domain}
		\positivepart{\tilde{\vortex}_\epsilon}(x)
		\positivepart{\tilde{\vortex}_\epsilon}(y)
		\depth(x)\regulargreenlaplace(x,y)\mumeasure\mumeasure(x,y)
	\\ - \iint\limits_{\domain\times\domain}
		\negativepart{\tilde{\vortex}_\epsilon}(x)
		\negativepart{\tilde{\vortex}_\epsilon}(y)
		\depth(x)\regulargreenlaplace(x,y)\mumeasure\mumeasure(x,y)
	\\- \iint\limits_{\domain\times\domain}
			\positivepart{\tilde{\vortex}_\epsilon}(x)
			\negativepart{\tilde{\vortex}_\epsilon}(y)
		\big(\depth(x)+\depth(y)\big)\greenlaplace(x,y)\mumeasure\mumeasure(x,y)
	- \ncste\vortexstrength_\epsilon^2 .
	\end{multline*}
Since $\regulargreenlaplace$ is a positive function, we also have
	\begin{multline*}
		\energy_\epsilon(\tilde{\vortex}_\epsilon) \geq
		\int_\domain\positivepart{\tilde{\vortex}_\epsilon}
			\positivefirstorderflow{\tilde{\vortex}_\epsilon}\mumeasure
		+ \int_\domain\negativepart{\tilde{\vortex}_\epsilon}
			\negativefirstorderflow{\tilde{\vortex}_\epsilon}\mumeasure
	\\ - \iint\limits_{\domain\times\domain}
		\positivepart{\tilde{\vortex}_\epsilon}(x)
		\positivepart{\tilde{\vortex}_\epsilon}(y)
		\depth(x)\regulargreenlaplace(x,y)\mumeasure\mumeasure(x,y)
	\\ - \iint\limits_{\domain\times\domain}
		\negativepart{\tilde{\vortex}_\epsilon}(x)
		\negativepart{\tilde{\vortex}_\epsilon}(y)
		\depth(x)\regulargreenlaplace(x,y)\mumeasure\mumeasure(x,y)
	\\- \iint\limits_{\domain\times\domain}
			\positivepart{\tilde{\vortex}_\epsilon}(x)
			\negativepart{\tilde{\vortex}_\epsilon}(y)
		\frac{\depth(x)+\depth(y)}{2\pi}
		\log\frac{\diameter(\domain)}{|x-y|}
		\mumeasure\mumeasure(x,y)
	- \ncste\vortexstrength_\epsilon^2 .
	\end{multline*}
By constructions of $\{X_\epsilon^\star:\epsilon>0\}$
and $\{Y_\epsilon^\star:\epsilon>0\}$ together with
\cref{regularestimate}, we have, for all sufficiently small $\epsilon>0$:
	\[\iint\limits_{\domain\times\domain}
		\positivepart{\tilde{\vortex}_\epsilon}(x)
		\positivepart{\tilde{\vortex}_\epsilon}(y)
		\depth(x)\regulargreenlaplace(x,y)\mumeasure\mumeasure(x,y)
	\leq \frac{\sup_\domain\depth}{2\pi\alpha}\tau^2\vortexstrength_\epsilon^2
		\log\log\frac{1}{\epsilon} + \ncste\vortexstrength_\epsilon^2 ,\]
and similarly
	\[\iint\limits_{\domain\times\domain}
		\negativepart{\tilde{\vortex}_\epsilon}(x)
		\negativepart{\tilde{\vortex}_\epsilon}(y)
		\depth(x)\regulargreenlaplace(x,y)\mumeasure\mumeasure(x,y)
	\leq \frac{\sup_\domain\depth}{2\pi\alpha}(1-\tau)^2\vortexstrength_\epsilon^2
		\log\log\frac{1}{\epsilon} + \ncste\vortexstrength_\epsilon^2 \]
and
	\begin{multline*}
	\iint\limits_{\domain\times\domain}
			\positivepart{\tilde{\vortex}_\epsilon}(x)
			\negativepart{\tilde{\vortex}_\epsilon}(y)
		\frac{\depth(x)+\depth(y)}{2\pi}
		\log\frac{\diameter(\domain)}{|x-y|}
		\mumeasure\mumeasure(x,y)
	\\ \leq \frac{\sup_\domain\depth}{2\pi\alpha}\big(2\tau(1-\tau)\big)
		\vortexstrength_\epsilon^2
		\log\log\frac{1}{\epsilon} + \ncste\vortexstrength_\epsilon^2 .
	\end{multline*}
From these estimates we conclude that for all sufficiently small $\epsilon>0$,
we have
	\begin{equation}\label{energyBelowEstimate}
		\energy_\epsilon(\tilde{\vortex}_\epsilon) \geq
		\int_\domain\positivepart{\tilde{\vortex}_\epsilon}
			\positivefirstorderflow{\tilde{\vortex}_\epsilon}\mumeasure
		+ \int_\domain\negativepart{\tilde{\vortex}_\epsilon}
			\negativefirstorderflow{\tilde{\vortex}_\epsilon}\mumeasure
	- \frac{\sup_\domain\depth}{2\pi\alpha}
		\vortexstrength_\epsilon^2\log\log\frac{1}{\epsilon}
	- \ncste\vortexstrength_\epsilon^2 .
	\end{equation}
Combining estimates~\eqref{energyUpperEstimate} and~\eqref{energyBelowEstimate} yields to
	\begin{align*}
		\iint\limits_{\domain\times\domain}
		\positivepart{\vortex_\epsilon}(x)
	&	\positivepart{\vortex_\epsilon}(y)
		\depth(x)\regulargreenlaplace(x,y)\mumeasure\mumeasure(x,y)
	\\&\quad+\iint\limits_{\domain\times\domain}
		\negativepart{\vortex_\epsilon}(x)
		\negativepart{\vortex_\epsilon}(y)
		\depth(x)\regulargreenlaplace(x,y)\mumeasure\mumeasure(x,y)
	\\&\quad+\iint\limits_{\domain\times\domain}
			\positivepart{\vortex_\epsilon}(x)
			\negativepart{\vortex_\epsilon}(y)
		\big(\depth(x)+\depth(y)\big)\greenlaplace(x,y)\mumeasure\mumeasure(x,y)
	\\&\qquad\qquad \leq \frac{\sup_\domain\depth}{2\pi\alpha}
		\vortexstrength_\epsilon^2\log\log\frac{1}{\epsilon}
		+ \ncste\vortexstrength_\epsilon^2 .
	\end{align*}
Let $\kappa\in(0,1)$. Taking advantage of the positivity of both $\greenlaplace$ and its regular part
$\regulargreenlaplace$, we have for all $\epsilon>0$ sufficiently small, since $\vortex_\epsilon$ concentrates on points
of maximal depth:
	\begin{multline}\label{eqnref2}
		\iint\limits_{\domain\times\domain}
		\positivepart{\vortex_\epsilon}(x)
		\positivepart{\vortex_\epsilon}(y)
		\regulargreenlaplace(x,y)\mumeasure\mumeasure(x,y)
	+\iint\limits_{\domain\times\domain}
		\negativepart{\vortex_\epsilon}(x)
		\negativepart{\vortex_\epsilon}(y)
		\regulargreenlaplace(x,y)\mumeasure\mumeasure(x,y)
	\\+2\iint\limits_{\domain\times\domain}
			\positivepart{\vortex_\epsilon}(x)
			\negativepart{\vortex_\epsilon}(y)
		\greenlaplace(x,y)\mumeasure\mumeasure(x,y)
	\leq \frac{\kappa^{-1}}{2\pi\alpha}
		\vortexstrength_\epsilon^2\log\log\frac{1}{\epsilon}
		+ \ncste\vortexstrength_\epsilon^2 .
	\end{multline}
Estimate~\eqref{eqnref2} gives us a bound for each of the three left terms.
Using \cref{regularestimate} and the positivity of $\greenlaplace$ and its regular part $\regulargreenlaplace$,
we first obtain for all sufficiently small $\epsilon>0$:
	\[ \tau^2\vortexstrength_\epsilon^2
		\log\frac{1}{\distance(X_\epsilon,\boundary\domain)}
	\leq \frac{2-\kappa}{\kappa\alpha}
		\vortexstrength_\epsilon^2\log\log\frac{1}{\epsilon}
	\]
and
	\[ (1-\tau)^2\vortexstrength_\epsilon^2
		\log\frac{1}{\distance(Y_\epsilon,\boundary\domain)}
	\leq \frac{2-\kappa}{\kappa\alpha}
		\vortexstrength_\epsilon^2\log\log\frac{1}{\epsilon}.
	\]
This proves that, for sufficiently small $\epsilon>0$
depending on $\kappa\in(0,1)$, we have:
	\[ \distance(X_\epsilon,\boundary\domain)
		\gtrsim \ncste\Bigg( \frac{1}{\log\frac{1}{\epsilon}} \Bigg)^{\frac{2-\kappa}{\kappa\alpha\tau^2}} ,
	\qquad \distance(Y_\epsilon,\boundary\domain)
		\gtrsim \ncste\Bigg( \frac{1}{\log\frac{1}{\epsilon}} \Bigg)^{\frac{2-\kappa}{\kappa\alpha(1-\tau)^2}}  .\]
The constants here may depend on $\kappa,\tau,\alpha,\depth$ and $\domain$,
but are independent of $\epsilon>0$. Now to conclude the proof, one observes that
we still have, from estimate~\eqref{eqnref2} and the positivity of $\regulargreenlaplace$:
	\[ 2\iint\limits_{\domain\times\domain}
			\positivepart{\vortex_\epsilon}(x)
			\negativepart{\vortex_\epsilon}(y)
		\greenlaplace(x,y)\mumeasure\mumeasure(x,y)
	\leq \frac{\kappa^{-1}}{2\pi\alpha}
		\vortexstrength_\epsilon^2\log\log\frac{1}{\epsilon}
		+ \ncste\vortexstrength_\epsilon^2 ,\]
which may be rewritten as
	\begin{multline*}
	2\iint\limits_{\domain\times\domain}
			\positivepart{\vortex_\epsilon}(x)
			\negativepart{\vortex_\epsilon}(y)
		\log\frac{\diameter(\domain)}{|x-y|}\mumeasure\mumeasure(x,y)
	\\ \leq2\iint\limits_{\domain\times\domain}
			\positivepart{\vortex_\epsilon}(x)
			\negativepart{\vortex_\epsilon}(y)
		\regulargreenlaplace(x,y)\mumeasure\mumeasure(x,y)
	+ \frac{\kappa^{-1}}{\alpha}
		\vortexstrength_\epsilon^2\log\log\frac{1}{\epsilon}
		+ \ncste\vortexstrength_\epsilon^2 .
	\end{multline*}
Using \cref{regularestimate} and the previous estimates on the distance at
the boundary, we obtain for sufficiently small $\epsilon>0$:
	\begin{align*}
	\vortexstrength_\epsilon^2\log\frac{1}{\distance(X_\epsilon,Y_\epsilon)}
	&\leq \vortexstrength_\epsilon^2
		\log\frac{1}{\max\{\distance(Y_\epsilon,\boundary\domain),
			\distance(X_\epsilon,\boundary\domain)\}}
	+ \frac{2-\kappa}{2\tau(1-\tau)\kappa\alpha}
		\vortexstrength_\epsilon^2\log\log\frac{1}{\epsilon} 
	\\&\leq \vortexstrength_\epsilon^2
		\log\Bigg(\ncste\bigg(\log\frac{1}{\epsilon}\bigg)^{\frac{2-\kappa}{\kappa\alpha}		\big(\frac{1}{\min\{\tau^2,(1-\tau)^2\}}+\frac{1}{2\tau(1-\tau)}\big)}\Bigg),
	\end{align*}
which gives us the a desired estimate
	\[ \distance(X_\epsilon,Y_\epsilon)
		\gtrsim \ncste\Bigg( \frac{1}{\log\frac{1}{\epsilon}} \Bigg)^{\frac{2-\kappa}{\kappa\alpha}\big(\frac{1}{\min\{\tau^2,(1-\tau)^2\}}+\frac{1}{2\tau(1-\tau)}\big)} .\qedhere\]
\end{proof}

\subsection{Accurate localization rule}
In this section we prove an accurate localization rule when
$\depth$ admits two or more maximizers inside $\domain$. In such situation,
we prove that the vortex pair always separates, and never reaches the boundary
$\boundary \domain$. We are not going to use the result of this section in the remaining part of the text.
Rather, we present the results because we think it draws an interesting link with the localization of
vortex pairs for the 2D~Euler equations.
\begin{lemma}\label{iiRepulsion}
Assume that $\depth$ admits at least two maximizers $X^\star,Y^\star$ in $\domain$.
Then we have
	\[ \liminf_{\epsilon\to 0}\distance\big(\{\vortex_\epsilon\neq 0\}, \boundary\domain \big) > 0 ,\]
and
	\[ \liminf_{\epsilon\to 0}\distance\Big( \{\positivepart{\vortex_\epsilon}>0\} , \{\negativepart{\vortex_\epsilon}> 0 \Big) > 0 .\]
\end{lemma}
\begin{proof}
Let us consider, for sufficiently small $\epsilon>0$, the function
	\[ \vortex_\epsilon^\star=\symmetrizearoundpoint{X^\star}{\positivepart{\vortex_\epsilon}} - \symmetrizearoundpoint{Y^\star}{\negativepart{\vortex_\epsilon}} .\]
The function $\vortex_\epsilon^\star$ is a $\mu$-rearrangement of $\vortex_\epsilon$, and
in particular we must have
$\energy_\epsilon(\vortex_\epsilon^\star)\leq\energy_\epsilon(\vortex_\epsilon)$.
Now one just should consider the same estimates than those made in the proof
of \cref{repulsionLog}. The $\log\log\frac{1}{\epsilon}$ term from equation~\eqref{energyBelowEstimate}
would reduce to a constant error term, and the conclusion would directly follow.
\end{proof}
\begin{proposition}\label{accurateRule}
Let $F:\domain\times\domain\to\reals$ be the function defined by
	\[ F(x,y) = -\depth(x)\regulargreenlaplace(x,y) + \rectifykernel(x,y)
		+ \sum^{\nbislands}_{i=0}(\staticflow_i(x)-\circulation_i)\Big[
		\mathcal{A}^{-1} \big[\staticflow_j(y)-\circulation_j\big]_{0\leq j\leq \nbislands}
		\Big]_i ,\]
and let $G:\domain\times\domain$ be the function defined by
	\[ G(x,y) = \frac{\depth(x)+\depth(y)}{2\pi}\log\frac{\diameter(\domain)}{|x-y|} .\]
Assume that $\depth$ admits at least two maximizers $X,Y$ in $\domain$ such that
	\[ \liminf_{\epsilon\to 0}\distance\big( \{\positivepart{\vortex_\epsilon}>0 \}, X \big) = 0 ,
		\qquad \liminf_{\epsilon\to 0}\distance\big( \{\negativepart{\vortex_\epsilon}>0 \}, Y \big) = 0 . \]
Then $(X,Y)$ minimizes the function
	\begin{align*}
		(x,y)&\in\domain\times\domain
		\\&\mapsto \bigg[\tau(1-\tau)G(x,y)
		-\tau^2F(x,x) - (1-\tau)^2F(y,y) +\tau(1-\tau)\big(F(x,y)+F(y,x)\big)\bigg]
	\end{align*}
over the set $\big(\domain\cap\{\depth=\sup_\domain\depth\}\big)\times\big(\domain\cap\{\depth=\sup_\domain\depth\}\big)$.
\end{proposition}
It would be interesting to derive an analogous for \cref{accurateRule} in the situation
where the maximizers of $\depth$ are not necessarily inside $\domain$.
Since we always have
	\[ \tau^2 + (1-\tau)^2 \geq 2\tau(1-\tau) ,\]
we would expect that the pair first tries to remain far form $\boundary\domain$
\emph{before} trying to be separated. However, a close inspection of
the proofs shows that the precise modulus of continuity of $\depth$
comes into play. We think that a further analysis of the auxiliary
function $\rectifykernel$ would be suitable to answer that question.
\begin{proof}[Proof of \cref{accurateRule}]
Let $X^\star,Y^\star\in\domain$ be such that $\depth(X^\star)=\sup_\domain\depth=\depth(Y^\star)$ and $X^\star\neq Y^\star$.
We consider $\epsilon>0$ so small that the function
	\[ \vortex^\star_\epsilon
		= \symmetrizearoundpoint{X^\star}{\positivepart{\vortex_\epsilon}}
		- \symmetrizearoundpoint{Y^\star}{\negativepart{\vortex_\epsilon}} \]
is a $\mu$-rearrangement of $\vortex_\epsilon$. In particular we have
$\energy_\epsilon(\vortex_\epsilon^\star)\leq\energy_\epsilon(\vortex_\epsilon)$.
We are going to expand these energies. Let us write in all generality
	\begin{multline*}
		\energy_\epsilon(\vortex^\star_\epsilon)
			= \int_\domain\positivepart{\vortex^\star_\epsilon}
				\positivefirstorderflow{\vortex^\star_\epsilon}\mumeasure
			+ \int_\domain\negativepart{\vortex^\star_\epsilon}
				\negativefirstorderflow{\vortex^\star_\epsilon}\mumeasure
			\\-\iint_{\domain\times\domain}G\ 
				\positivepart{\vortex^\star_\epsilon}
				\negativepart{\vortex^\star_\epsilon}
			\mumeasure\mumeasure(x,y)
			+ \iint_{\domain\times\domain}F\ 
				\vortex^\star_\epsilon\vortex^\star_\epsilon
				\mumeasure\mumeasure ,
	\end{multline*}
where $F$ and $G$ are defined as in the statement, and recall that since $(\domain,\depth)$ is a continuous lake,
the function $F$ is continuous on $\domain\times\domain$.
Because both $\{\vortex^\star_\epsilon:\epsilon>0\}$ and $\{\vortex_\epsilon:\epsilon>0\}$
concentrate, we have
	\begin{multline*}
	\liminf_{\epsilon\to 0}\frac{1}{\vortexstrength_\epsilon^2}
		\iint_{\domain\times\domain}F\ \vortex_\epsilon\vortex_\epsilon
		\mumeasure\mumeasure
	\\= \tau^2F(X,X) + (1-\tau)^2F(Y,Y)
		- \tau(1-\tau)\big(F(X,Y)+F(Y,X)\big) ,
	\end{multline*}
while
	\begin{multline*}
	\liminf_{\epsilon\to 0}\frac{1}{\vortexstrength_\epsilon^2}
		\iint_{\domain\times\domain}F\ \vortex^\star_\epsilon\vortex^\star_\epsilon
		\mumeasure\mumeasure
	\\= \tau^2F(X^\star,X^\star) + (1-\tau)^2F(Y^\star,Y^\star)
		- \tau(1-\tau)\big(F(X^\star,Y^\star)+F(Y^\star,X^\star)\big) .
	\end{multline*}
On the other hand,
because $X^\star$ and $Y^\star$ are exact maximizers of $\depth$,
we may apply Riesz-Sobolev rearrangement inequality, \cref{RieszSobolev},
to see that there exists $\error_\epsilon>0$ with
$\lim\limits_{\epsilon\to 0}\error_\epsilon = 0$,
and such that
	\begin{multline*}
		\energy_\epsilon(\vortex^\star_\epsilon)
			= \big(\sup_\domain\depth\big)
				\int_\domain\positivepart{\vortex^\star_\epsilon}
				\depth^{-1}\positivefirstorderflow{\vortex^\star_\epsilon}\mumeasure
			+ \big(\sup_\domain\depth\big)
				\int_\domain\negativepart{\vortex^\star_\epsilon}
				\depth^{-1}\negativefirstorderflow{\vortex^\star_\epsilon}\mumeasure
			\\+\iint_{\domain\times\domain}\big(\depth(x)+\depth(y)\big)
				\greenlaplace(x,y)
				\positivepart{\vortex^\star_\epsilon}(x)
				\negativepart{\vortex^\star_\epsilon}(y)
			\mumeasure\mumeasure(x,y)
			\\+ \iint_{\domain\times\domain}F(x,y)
				\vortex^\star_\epsilon(x)\vortex^\star_\epsilon(y)
				\mumeasure\mumeasure(x,y)
			+ \error_\epsilon\vortexstrength_\epsilon^2 .
	\end{multline*}
This yields to the inequality
	\begin{multline*}
		\frac{1}{\vortexstrength_\epsilon^2}\iint_{\domain\times\domain}G\ 
				\positivepart{\vortex_\epsilon}
				\negativepart{\vortex_\epsilon}
			\mumeasure\mumeasure
		-\frac{1}{\vortexstrength_\epsilon^2}
		\iint_{\domain\times\domain}F\ \vortex_\epsilon\vortex_\epsilon
		\mumeasure\mumeasure
		\\\leq\frac{1}{\vortexstrength_\epsilon^2}\iint_{\domain\times\domain}G\ 
				\positivepart{\vortex^\star_\epsilon}
				\negativepart{\vortex^\star_\epsilon}
			\mumeasure\mumeasure
		-\frac{1}{\vortexstrength_\epsilon^2}
		\iint_{\domain\times\domain}F\ \vortex^\star_\epsilon\vortex^\star_\epsilon
		\mumeasure\mumeasure  +\error_\epsilon
	\end{multline*}
This latter inequality holds for all couple of points $(X^\star,Y^\star)\in\domain\times\domain$ with $X^\star\neq Y^\star$ and both $X^\star,Y^\star$ maximize $\depth$.
Letting $\epsilon\to 0$ and applying Fatou's lemma on the left hand side,
we obtain the desired conclusion.
\end{proof}

\section{Asymptotic shape}

In this section we prove that the vortex pair asymptotically looks like
two symmetric functions. We base our analysis on results obtained by Burchard~\&~Guo~\cite{BurchardGuo}*{Lemma~3.2}
on the asymptotic shape of asymptotic maximizers of singular integrals.
More precisely, our goal is to prove that
	\[ \limsup_{\epsilon\to 0}\bigg\{ \energy_\epsilon\big(\vortex_\epsilon\big)
		- \energy_\epsilon\big( \vortex_\epsilon^\star\big)\bigg\} = 0 ,\]
for some radially symmetric $\mu$-rearrangement $\vortex_\epsilon^\star$ of the energy maximizer
$\vortex_\epsilon$. Then we exploit this information to obtain convergence in shape of the
rescaled maximizers.

\subsection{Preliminary study of oscillations}

The aim of this paragraph is to prove the following asymptotic limit estimate:
\begin{proposition}\label{smallOsc}
For all $X_\epsilon\in\{\positivepart{\vortex_\epsilon}>0\}$, let
$\vortex^\star_\epsilon=\symmetrizearoundpoint{X_\epsilon}{\positivepart{\vortex_\epsilon}}
		- \negativepart{\vortex_\epsilon}$. We have
	\begin{multline*}
		\liminf_{\epsilon\to0}\Bigg\{
		\iint_{\domain\times\domain}F(x,y)\vortex_\epsilon(x)\vortex_\epsilon(y)\mumeasure\mumeasure(x,y)
		-\iint_{\domain\times\domain}G(x,y)\positivepart{\vortex_\epsilon}(x)\negativepart{\vortex_\epsilon}(y)\mumeasure\mumeasure(x,y)
			\\- \iint_{\domain\times\domain}F(x,y)\vortex^\star_\epsilon(x)\vortex^\star_\epsilon(y)\mumeasure\mumeasure(x,y)
		+\iint_{\domain\times\domain}G(x,y)\positivepart{\vortex^\star_\epsilon}(x)\negativepart{\vortex^\star_\epsilon}(y)\mumeasure\mumeasure(x,y)
		\Bigg\} = 0 .
	\end{multline*}
\end{proposition}
The following a priori estimate from elliptic regularity theory is useful:
\begin{proposition}[\cite{GilbardTrudinger}*{Theorem~8.22}]\label{slowOscillation}
Let $\mathcal{U}\subseteq\domain$ be such that $\inf_{\mathcal{U}}\depth>0$.
Let $f_1,f_2\in L^q(\mathcal{U},\lebesguemeasure)$ for some $q>2$,
and $\psi\in W^{1,2}(\mathcal{U})$ be such that, for all $\varphi\in C^1_c(\mathcal{U})$, we have
	\[ \int_{\mathcal{U}}\scalarproduct{\gradient\psi}{\gradient\varphi}{\plane}\ \frac{\dif x}{\depth}
		= \int_{\mathcal{U}}\scalarproduct{(f_1,f_2)}{\gradient\varphi}\ \dif\lebesguemeasure .\]
For all $\gamma,\varsigma>0$, there exists $\sigma>0$ and a constant $C>0$ such that for all $\epsilon>0$
and for all $y\in\mathcal{U}$ with $\distance(y,\boundary\domain)\geq \big(\log\frac{1}{\epsilon}\big)^\gamma$, we have
	\[ \sup_{x\in B(y,\epsilon^\varsigma)}\Big\{ \big|\psi(x)-\psi(y)\big| \Big\}
		\leq C\epsilon^\sigma\bigg( \norm{\psi}{L^\infty(\mathcal{U})} + \norm{(f_1,f_2)}{L^q} \bigg) .\]
\end{proposition}
Since we already know that the maximizing vortex pair $\vortex_\epsilon$ concentrates on a point
of maximal depth, we may apply \cref{slowOscillation} to control the oscillation of each $\rectifykernel_y$,
on the vortex core. Also, \cref{slowOscillation} for $b\equiv 1$ yields an estimate for the regular part of the
Green's function $\regulargreenlaplace$. Finally, we also recall that the function (see~\cite{myArticle})
	\[ y\in\closure{\domain}\mapsto \norm{\rectifykernel(\cdot,y)}{L^\infty} \]
is continuous on $\closure{\domain}$, since $(\domain,\depth)$ is a continuous lake by assumption.
The vortex $\positivepart{\vortex_\epsilon}$
and $\negativepart{\vortex_\epsilon}$ both concentrate with diameter of a priori order $\epsilon^\varsigma$,
while their distance from each others and from the boundary $\boundary\domain$
is of not smaller than some order $\big(\log\frac{1}{\epsilon}\big)^\gamma$.
From these a priori results and a direct application of
\cref{slowOscillation}, one can easily deduce \cref{smallOsc}. We skip the proof.

\subsection{Scaling process and energy convergence}

We fix a family of points $\big\{X_\epsilon\in\{\positivepart{\vortex_\epsilon}>0\}:\epsilon>0\big\}$, and we recall that
	\[ \{\positivepart{\vortex_\epsilon}>0\} \subseteq B(X_\epsilon,\epsilon^\varsigma) ,\]
for some $\varsigma>0$ independent of $\epsilon>0$.
Our analysis is made for the positive part of the vortex pair,
but similar results hold for the negative part as well.

Given a positive and measurable function $f:\domain\to\reals^+$ with finite $\mu$-integral, we define its
scaled version as the function
	\[ \scale{f} : \plane\to\reals^+ : \scale{f}(x) = \frac{\epsilon}{\int_\domain f\mumeasure}\ f\big(\epsilon\,x + X_\epsilon\big) .\]
The rescaled version of the measure $\mu$ is $\scalemumeasure$ defined as
	\[ \int_{\plane}g(x)\scalemumeasure(x) = \frac{1}{\epsilon}\int_\domain g\big(\epsilon^{-1}\,(x - X_\epsilon)\big)\mumeasure(x) .\]
By construction, we have
	\[ \int_{\plane}\scale{f}\scalemumeasure = 1 ,\]
while
	\[ \norm{\scale{f}}{L^p_{\scalemumeasure}} = \frac{\norm{f}{L^p_{\mu}}\epsilon^{2\big(1-\frac{1}{p}\big)}}{\int_\domain f\mumeasure} .\]
Observe that the scaling process depends on $\epsilon>0$ and on $X_\epsilon$,
although we do not explicitly mention this dependence
in our notations.
\begin{lemma}
Let $f_\epsilon=\scale{\positivepart{\vortex_\epsilon}}$ and $f^\star_\epsilon=\scale{\big( \symmetrizearoundpoint{X_\epsilon}{\positivepart{\vortex_\epsilon}} \big)}$. We have
	\begin{multline*} \lim_{\epsilon\to 0} \Bigg\{
		\iint_{\plane\times\plane}\log\frac{1}{|x-y|}
			\,f_\epsilon^\star(x)f_\epsilon^\star(y)
			\dif(x,y)
	\\- \iint_{\plane\times\plane}\log\frac{1}{|x-y|}
			\,f_\epsilon(x)f_\epsilon(y)
			\dif(x,y)
	\Bigg\} = 0 .
	\end{multline*}
\end{lemma}
\begin{proof}
According to \cref{smallOsc} together with the energy maximization principle and the 
Riesz-Sobolev rearrangement inequality, \cref{RieszSobolev}, one directly has the asymptotic behavior
	\begin{multline*} \lim_{\epsilon\to 0} \Bigg\{
		\iint_{\domain\times\domain}\log\frac{\diameter(\domain)}{|x-y|}
			\,\symmetrizearoundpoint{X_\epsilon}{\positivepart{\vortex_\epsilon}}(x)\symmetrizearoundpoint{X_\epsilon}{\positivepart{\vortex_\epsilon}}(y)
			\mumeasure\mumeasure(x,y)
	\\- \iint_{\domain\times\domain}\log\frac{\diameter(\domain)}{|x-y|}
			\,\positivepart{\vortex_\epsilon}(x)\positivepart{\vortex_\epsilon}(y)
			\mumeasure\mumeasure(x,y)
	\Bigg\} = 0 .
	\end{multline*}
Since $\positivepart{\vortex_\epsilon}$ and $\symmetrizearoundpoint{X_\epsilon}{\positivepart{\vortex_\epsilon}}$
are $\mu$-rearrangements of each others, we also get
	\begin{multline*} \lim_{\epsilon\to 0} \Bigg\{
		\iint_{\domain\times\domain}\log\frac{\epsilon}{|x-y|}
			\,\symmetrizearoundpoint{X_\epsilon}{\positivepart{\vortex_\epsilon}}(x)\symmetrizearoundpoint{X_\epsilon}{\positivepart{\vortex_\epsilon}}(y)
			\mumeasure\mumeasure(x,y)
	\\- \iint_{\domain\times\domain}\log\frac{\epsilon}{|x-y|}
			\,\positivepart{\vortex_\epsilon}(x)\positivepart{\vortex_\epsilon}(y)
			\mumeasure\mumeasure(x,y)
	\Bigg\} = 0 .
	\end{multline*}
The scaling process yields to
	\begin{multline*} \lim_{\epsilon\to 0} \Bigg\{
		\iint_{\plane\times\plane}\log\frac{1}{|x-y|}
			\,f_\epsilon^\star(x)f_\epsilon^\star(y)
			\scalemumeasure\scalemumeasure(x,y)
	\\- \iint_{\plane\times\plane}\log\frac{1}{|x-y|}
			\,f_\epsilon(x)f_\epsilon(y)
			\scalemumeasure\scalemumeasure(x,y)
	\Bigg\} = 0 .
	\end{multline*}
Now observe that the diameter of $\{f_\epsilon>0\}$ is of order $\epsilon^{\varsigma-1}$,
while the diameter of $\{f^\star_\epsilon>0\}$ is of order $1$. Hence, since $\depth\in C^{0,\alpha}(\closure{\domain})$
is uniformly bounded on $\{\vortex_\epsilon>0\}$, independently of $\epsilon>0$ sufficiently small,
we obtain
	\begin{multline*} \lim_{\epsilon\to 0} \Bigg\{
		\iint_{\plane\times\plane}\log\frac{1}{|x-y|}
			\,f_\epsilon^\star(x)f_\epsilon^\star(y)
			\dif(x,y)
	\\- \iint_{\plane\times\plane}\log\frac{1}{|x-y|}
			\,f_\epsilon(x)f_\epsilon(y)
			\dif(x,y)
	\Bigg\} = 0 .
	\end{multline*}
\end{proof}

\subsection{Criterion for convergence}

We require the following sufficient criterion to prove convergence in measure of a sequence of functions.
The distribution of a positive and measurable function $f:\domain\to\reals$ is defined by
	\[ \distribution{f} : \reals^+\to\reals^+\cup\{+\infty\} : \distribution{f}(t) = \mu\big(\{f>t\}\big) .\]
\begin{lemma}[Criterion for convergence in measure]\label{criterionconvergence}
Let $(f_n)_{n\in\integers}$ be a family of measurable, positive and real valued functions,
and $f$ be a measurable positive and real valued function. We assume that
	\[ \distribution{f}(0) < +\infty \qquad\text{and}\qquad\ \lim_{M\to+\infty}\distribution{f}(M) = 0 .\]
If the sequence $\big(\distribution{f_n}\big)_{n\in\integers}$ converges pointwisely to $\distribution{f}$,
and if for $\mu$-almost-all every $t\geq 0$, we have
	\[ \lim_{n\to\infty}\mu\big( \{f_n>t\}\Delta \{f>t\} \big) = 0 ,\]
then $(f_n)_{n\in\integers}$ converges in $\mu$-measure to $f$.
\end{lemma}
This result is standard, but we have not found it in the literature.
\begin{proof}
Let us fix $s>0$. If there exists $x\in\{f_n>f+s\}$, choose $M=M(x,s)>0$ such that $f(x)\leq M$.
For all $t>0$, we define
	\[ A_t = \{f_n>t+s/2\} \setminus \{f>t\} \]
and we compute
	\begin{align*}
		\int^M_0\chi_{A_t}(x)\dif t
			&= \int^{f(x)}_0\chi_{A_t}(x)\dif t + \int^{f_n(x)-\frac{s}{2}}_{f(x)}\chi{A_t}(x)\dif t + \int^M_{f_n(x)-\frac{s}{2}}\chi_{A_t}(x)\dif t
			\\&= \int^{f_n(x)-\frac{s}{2}}_{f(x)}1\dif t = f_n(x) - \frac{s}{2} - f(x)
			\\&\geq \frac{s}{2}\chi_{\{f_n\geq f+s\}}(x) ,
	\end{align*}
and therefore
	\[ \chi_{\{f_n\geq f+s\}}(x) \leq \frac{2}{s}\int^M_0\chi_{\{f_n>t+s/2\}\setminus\{f>t\}}(x)\dif t + \chi_{\{f>M\}}(x) .\]
The last inequality extends for all $x\in\domain$ and for all $M>0$.
It also holds if $\{f_n>f+s\}$ is the empty set. By symmetry, we prove similarly the for all $s>0$, for all $M>0$ and for all $x\in\domain$:
	\[ \chi_{\{f\geq f_n+s\}}(x) \leq \frac{2}{s}\int^M_0\chi_{\{f>t+s/2\}\setminus\{f_n>t\}}(x)\dif t + \chi_{\{f_n>M\}}(x) .\]
Since we have
	\[ \{f_n>t+s/2\}\setminus\{f>t\} \ \cup\ \{f>t+s/2\}\setminus\{f_n>t\} \subseteq \{ f>t\} \Delta \{f_n>t\} ,\]
we obtain
	\[ \chi_{\{|f_n-f|\geq s\}}(x) \leq \frac{2}{s}\int^M_0\chi_{\{ f>t\} \Delta \{f_n>t\}}(x)\dif t + \chi_{\{f>M\}}(x) + \chi_{\{f_n>M\}}(x) .\]
In particular, one concludes by Tonelli's theorem for all $s>0$ and for all $M>0$, we have
	\[ \mu\big(\{|f_n-f|\geq s\}\big) \leq \frac{2}{s}\int^M_0\mu\big(\{ f>t\} \Delta \{f_n>t\}\big)\dif t
		+ \mu\big(\{f>M\}\big) + \mu\big( \{f_n>M\} \big) .\]
As a corollary of our assumptions, we have
	\[ \limsup_{n\to +\infty}\mu\big(\{f_n>0\}\big) < +\infty ,\]
and thus one may apply the Lebesgue's dominated convergence theorem and the condition $\lim\limits_{M\to+\infty}\distribution{f}(M)=0$
to obtain the conclusion.
\end{proof}
\begin{remark}
The condition is actually an equivalent condition, but we are not going to use the
counterpart in this text.
\end{remark}

\subsection{Asymptotic shape}

\begin{lemma}
Let $f_\epsilon=\scale{\positivepart{\vortex_\epsilon}}$ and let us denote by
$f^\Delta_\epsilon$ the symmetric radially nonincreasing
Lebesgue-rearrangement of $f_\epsilon$ centered on $0$. We have
	\begin{multline*} \liminf_{\epsilon\to 0} \Bigg\{
		\iint_{\plane\times\plane}\log\frac{1}{|x-y|}
			\,f^\Delta_\epsilon(x)f^\Delta_\epsilon(y) \dif(x,y)
	\\- \iint_{\plane\times\plane}\log\frac{1}{|x-y|}
			\,f_\epsilon(x)f_\epsilon(y) \dif(x,y)
	\Bigg\} = 0 .
	\end{multline*}
\end{lemma}
\begin{proof}
Let $f^\star_\epsilon=\scale{\big( \symmetrizearoundpoint{X_\epsilon}{\positivepart{\vortex_\epsilon}} \big)}$.
For all $\epsilon>0$ sufficiently small, the functions $f^\star_\epsilon$ and $f^\Delta_\epsilon$
are both supported in some ball $B(0,R)$ with $R>0$ independent of $\epsilon>0$.
Furthermore, they belong to $L^p\big( B(0,R),\lebesguemeasure\big)$.
By scaling process, it is straightforward to see that
	\[ \lim_{\epsilon\to 0}\bigg\{ \lebesguemeasure\Big(\{x\in B(0,R):f^\star_\epsilon(x)>t\}\Big)
		- \lebesguemeasure\Big(\{x\in B(0,R):f^\Delta_\epsilon(x)>t\}\Big) \bigg\} = 0 . \]
According to \cref{criterionconvergence}, we conclude that a function $g$ is an accumulation point in the sense of
convergence in Lebesgue measure for the family $\{f_\epsilon^\Delta:\epsilon>0\}$ if, and only if, it is also
an accumulation point in the sense of convergence in Lebesgue measure
for the family $\{f_\epsilon^\star:\epsilon>0\}$.

According to Helly's selection principle, the family $\{f^\Delta_\epsilon:\epsilon>0\}$
admits at least one accumulation point (in the sense of convergence in Lebesgue measure)
$g\in L^p\big(B(0,R),\lebesguemeasure\big)$
that is symmetric radially nonincreasing. Up to taking a new subsequence, and because the set
$\{f^\Delta_\epsilon:\epsilon>0\}$ is bounded in $L^p\big(B(0,R),\lebesguemeasure\big)$, one may
assume that the convergence occurs weakly in $L^p\big(B(0,R),\lebesguemeasure\big)$.
To avoid heavy notations, we assume that
	\[ f^\Delta_\epsilon\to g\ \text{weakly and in convergence in Lebesgue measure as}\ \epsilon\to 0 .\]
We deduce that
	\[ f^\star_\epsilon\to g\ \text{weakly and in convergence in Lebesgue measure as}\ \epsilon\to 0 .\]
Let us write, for all $u\in L^p(\plane,\lebesguemeasure)$, 
	\[ -\log\ast\ u : \plane\to\reals : -\log\ast\ u(y) = \int_{\plane}u(x)\log\frac{1}{|x-y|}\dif x .\]
We have $-\log\ast\ u\in W^{1,p}\big(B(0,R)\big)$ whenever $u\in L^p(\plane,\lebesguemeasure)$~\cite{GilbardTrudinger}*{Theorem~9.9}.
Relying on Rellich-Kondrashov theorem, and up to taking a subsequence,
we may abuse of notations and assume that
	\[ -\log\ast\ f^\Delta_\epsilon \to -\log\ast\ g
		\ \text{weakly in}\ W^{1,p}(B(0,R))\ \text{and strongly in}\ L^p\big(B(0,R),\lebesguemeasure\big) ,\]
and
	\[ -\log\ast\ f^\star_\epsilon \to -\log\ast\ g
		\ \text{weakly in}\ W^{1,p}(B(0,R))\ \text{and strongly in}\ L^p\big(B(0,R),\lebesguemeasure\big) .\]
The conclusion now follows directly by a strong-weak convergence argument.
\end{proof}
Using a theory developed by Buchard~\&~Guo~\cite{BurchardGuo}*{Lemma3.2}, one may now assert the following optimal shape theorem:
\begin{theorem}
Every accumulation point of the family $\big\{\scale{\positivepart{\vortex_\epsilon}}:\epsilon>0\big\}$
in the sense of Lebesgue measure in $\plane$ may be written as $f^\star\circ T$, where $T$ is a translation in $\plane$
and $f^\star$ is a radially nonincreasing function.
In particular, there exists $R>0$ independent of $\epsilon>0$, and a family $\{X_\epsilon\in\domain:\epsilon>0\}$, such that
	\[ \lim_{\epsilon\to 0}\bigg\{
		\frac{1}{\tau\vortexstrength_\epsilon}\int_{\domain\setminus B(X_\epsilon,R\epsilon)}\positivepart{\vortex_\epsilon}\mumeasure
	\bigg\} = 0 .\]
\end{theorem}
We skip the proof of this result, since it is a direct copy of the proof of~\cite{BurchardGuo}*{Lemma3.2}.
Since we are working with the Newtonian potential kernel $K$, the proof
applies in $L^p(\plane,\lebesguemeasure)$ instead of $L^2(\plane,\lebesguemeasure)$.

\section{Application to a relaxed maximization problem}

In this last section, we apply the result previously obtained for a more common problem
that already appeared in the work of Turkington~\citelist{\cite{TurkingtonFriedman}\cite{TurkingtonSteady1}\cite{TurkingtonSteady2}}
for non sign changing vortex.
Let us consider the following energy maximization problem: Maximize the energy $\energy_\epsilon$ over the class
	\[ \Gamma_\epsilon = \Bigg\{ f:\domain\to\reals : |f|\leq \epsilon^{-2}\vortexstrength_\epsilon,
		\int_\domain\positivepart{f}\mumeasure=\tau\vortexstrength_\epsilon,\int_\domain\negativepart{f}\mumeasure=(1-\tau)\vortexstrength_\epsilon \Bigg\} .\]
This class is physically motivated by the fact that the integral constraints are constraints on the \emph{vortex strength},
while the $L^\infty(\domain)$ constraint is natural, since the vortex is a kinematic quantity.
Without too much efforts, we show that the results we have previously obtained may be applied to this problem.

\begin{theorem}\label{byproduct}
There exists at least one maximizer of $\energy_\epsilon$ over the class $\Gamma_\epsilon$.
Moreover, every maximizer $f^\star_\epsilon$ is a function of the form
	\[ f^\star_\epsilon = \epsilon^{-2}\vortexstrength_\epsilon\big( \chi_{A^+_\epsilon} - \chi_{A^-_\epsilon} \big) \]
with $A^+_\epsilon\cap A^-_\epsilon=\emptyset$ two distinct measurable sets such that
	\[ \mu\big( A^+_\epsilon \big) = \tau\epsilon^2 ,\quad \mu\big( A^-_\epsilon \big) = (1-\tau)\epsilon^2 ,\]
and there exist $X_\epsilon\in A^+_\epsilon$, $Y_\epsilon\in A^-_\epsilon$, such that
	\[ \lim_{\epsilon\to 0}
		\frac{\mu\big( A^+_\epsilon\Delta B(X_\epsilon,R\sqrt{\tau}\,\epsilon) \big)}{\epsilon^2} = 0 ,
	\qquad
	\lim_{\epsilon\to 0}
		\frac{\mu\big( A^-_\epsilon\Delta B(Y_\epsilon,R\sqrt{1-\tau}\,\epsilon) \big)}{\epsilon^2} = 0 ;\]
where $R^{-1}=\sqrt{\pi\sup_\domain\depth}$.
\end{theorem}
\begin{proof}
The existence part follows from usual functional analysis techniques. More precisely, the class $\Gamma_\epsilon$
is weakly compact in $L^p(\domain,\mu)$, for every $p>1$, while the energy $\energy_\epsilon$
is weakly continuous on bounded set in $L^p(\domain,\mu)$, for every $p>1$. 
The last assumptions of the theorem directly follows from our preliminary work. All it remains to
prove is that \emph{every} maximizer $f^\star$ of $\energy_\epsilon$ over $\Gamma_\epsilon$ is of the form
	\[ f^\star = \epsilon^{-2}\vortexstrength_\epsilon\big( \chi_A - \chi_B \big) \]
for some measurable disjoint sets $A,B\subseteq\domain$. By symmetry, it is sufficient to prove that we have $\positivepart{f^\star}=\epsilon^{-2}\vortexstrength_\epsilon\chi_A$.
We define the measure $\mu_\epsilon$ as
	\[ \mu_\epsilon(A) = \mu\big( A\cap \{f^\star\geq 0\} \big) ,\]
for all measurable set $A\subseteq\domain$. We define also the energy
	\[ \tilde{\energy}_\epsilon(f) = \energy_\epsilon\big(f-\negativepart{f^\star}\big) ,\]
and the new class
	\[ \tilde{\Gamma}_\epsilon = \Bigg\{ f: 0\leq f\leq \epsilon^{-2}\vortexstrength_\epsilon : \{f>0\}\cap\{f^\star<0\}=\emptyset,
		\int_\domain f\dif\mu_\epsilon=\tau\vortexstrength_\epsilon \Bigg\} .\]
Then clearly every function $f\in\tilde{\Gamma}_\epsilon$ induces a function $f-\negativepart{f^\star}\in\Gamma_\epsilon$,
so that
	\[ \tilde{\energy}_\epsilon(f) = \energy_\epsilon\big( f-\negativepart{f^\star}\big) \leq \energy_\epsilon(f^\star) = \tilde{\energy}_\epsilon\big(\positivepart{f^\star}\big) .\]
Therefore, the function $\positivepart{f^\star}\in\tilde{\Gamma}_\epsilon$ is a maximizer of $\tilde{\energy}_\epsilon$ on $\tilde{\Gamma}_\epsilon$.

Furthermore, one observes that a function $f$ belongs to $\tilde{\Gamma}_\epsilon$ if,
and only if we have
	\[ \int_\domain f\dif\mu_\epsilon = \tau\vortexstrength_\epsilon \]
and the additional property that, for all $t\geq 0$:
	\[ \int^t_0f^{\mu_\epsilon}(s)\dif s \leq \int^t_0\epsilon^{-2}\vortexstrength_\epsilon\chi_{[0,\tau\epsilon)}(s)\dif s ,\]
where $f^{\mu_\epsilon}\in L^\infty\big(\reals^+\big)$ denotes the non-decreasing $\mu_\epsilon$-rearrangement
of $f$ in the positive half line $\reals^+$. In other words, we have the set identity
	\begin{multline*}
		\tilde{\Gamma}_\epsilon
		= \Bigg\{ f\in L^\infty(\domain):0\leq f,\{f>0\}\subseteq\domain\setminus\{f^\star<0\},
	\\ \int_\domain f\dif\mu_\epsilon = \tau\vortexstrength_\epsilon,
		\forall t\geq 0: \int^t_0 f^{\mu_\epsilon}(s)\dif s \leq \int^t_0\epsilon^{-2}\vortexstrength_\epsilon\chi_{[0,\tau\epsilon)}(s)\dif s \Bigg\} .
	\end{multline*}
Since we also have
	\[ \tau\vortexstrength_\epsilon = \int_\domain\positivepart{f^\star}\dif\mu_\epsilon \leq \epsilon^{-2}\vortexstrength_\epsilon\dif\mu_\epsilon\big(\{f^\star>0\}\big) ,\]
there exists at least one measurable set $A_\epsilon\subseteq\domain\setminus\{f^\star<0\}$ such that
	\[ \mu\big( A_\epsilon \big) = \tau\epsilon^2 .\]
In particular, we have
$\big(\epsilon^{-2}\vortexstrength_\epsilon\chi_{A_\epsilon}\big)^{\mu_\epsilon}=\epsilon^{-2}\vortexstrength_\epsilon\chi_{[0,\tau\epsilon)}$,
the function $\epsilon^{-2}\vortexstrength_\epsilon\chi_{A_\epsilon}$ belongs to
$\tilde{\Gamma}_\epsilon$, and we have
	\begin{multline*}
		\tilde{\Gamma}_\epsilon
		= \Bigg\{ f\in L^\infty(\domain):0\leq f\leq \epsilon^{-2}\vortexstrength_\epsilon,\{f>0\}\subseteq\domain\setminus\{f^\star<0\},
	\\ \int_\domain f\dif\mu_\epsilon = \tau\vortexstrength_\epsilon,
		\forall t\geq 0: \int^t_0 f^{\mu_\epsilon}(s)\dif s \leq \int^t_0\big(\epsilon^{-2}\vortexstrength_\epsilon\chi_{A_\epsilon}\big)^{\mu_\epsilon}(s)\dif s \Bigg\} .
	\end{multline*}
The class $\tilde{\Gamma}_\epsilon$ equals the weak closure in $L^p(\domain,\mu_\epsilon)$
of the indicator function $\epsilon^{-2}\vortexstrength_\epsilon\chi_{A_\epsilon}$~\citelist{\cite{Ryff1}\cite{Ryff2}\cite{Douglas}}.
Since $\tilde{\energy}_\epsilon$ is a strictly convex functional,
every maximizer of $\tilde{\energy}_\epsilon$ over the convex set $\tilde{\Gamma}_\epsilon=\weakclosure{\rearrangement(\epsilon^{-2}\vortexstrength_\epsilon\chi_{A_\epsilon})}$
must be an extreme point of this set. On the other hand, the set of extreme points of $\weakclosure{\rearrangement(\epsilon^{-2}\vortexstrength_\epsilon\chi_{A_\epsilon})}$
is the set $\rearrangement(\epsilon^{-2}\vortexstrength_\epsilon\chi_{A_\epsilon})$~\citelist{\cite{Ryff1}\cite{Ryff2}\cite{Douglas}}.
From this it follows that $\positivepart{f^\star}$
is a $\mu_\epsilon$-rearrangement of $\epsilon^{-2}\vortexstrength_\epsilon\chi_{A_\epsilon}$.
\end{proof}

%\begin{center}\small
%\begin{tikzpicture}[scale=0.9]
%\begin{axis}[title={Domain with an exp-exp cusp}]
%\addplot[color=black] table {./graphs/cuspedDomain.txt};
%\end{axis}
%\end{tikzpicture}
%\end{center}

%%%%%%%%%%%%%%%%%%%%%%%%%%%%%%%%%%%%%%%%%%%%%%%%%%%%%%%%%%%%%%%%%%%%%%%%%%%%%%%%%%%%

\end{document}